\documentclass[11pt,reqno]{amsart}
\usepackage[utf8]{inputenc}

\usepackage[x11names]{xcolor}
\usepackage{geometry}
\usepackage{tikz-cd}
\usepackage{amsmath}
\usepackage{amsfonts, amssymb, amsthm}
\usepackage{mathtools}
\usepackage{enumitem}
\setenumerate{itemsep=1.5ex, topsep=2ex}
\usepackage{hyperref}
\hypersetup{
    colorlinks=true,
    linktoc=all,
    linkcolor=magenta,
    linktocpage,
    citecolor=DodgerBlue4
}

\usepackage[capitalise]{cleveref}

\newtheorem{theorem}{Theorem}[section]
\newtheorem{lemma}[theorem]{Lemma}
\newtheorem{corollary}[theorem]{Corollary}
\theoremstyle{definition}
\newtheorem{definition}[theorem]{Definition}
\newtheorem{proposition}[theorem]{Proposition}

\newtheorem{warning}[theorem]{Warning}

\newtheorem{observation}[theorem]{Observation}

\newtheorem{notation}[theorem]{Notation}

\newtheorem{construction}[theorem]{Construction}
\theoremstyle{remark}
\newtheorem{remark}[theorem]{Remark}

\newenvironment{diag}{\begin{tikzcd}}{\end{tikzcd}}

\newcommand*{\Fin}{\operatorname{\mathbf{Fin}}}
\newcommand*{\Set}{\operatorname{\mathbf{Set}}}
\newcommand*{\Mon}{\operatorname{\mathbf{Mon}}}
\newcommand*{\CAlg}{\operatorname{CAlg}}
\newcommand*{\Dalg}{\operatorname{DAlg}}

\newcommand*{\CGrp}{\operatorname{CGrp}}
\newcommand*{\Mod}{\operatorname{Mod}}

\newcommand*{\Poly}{\operatorname{\mathbf{Poly}}}
\newcommand*{\CR}{\operatorname{\mathbf{CR}}}
\newcommand*{\PolyFp}{\operatorname{Poly}_{\Fp}}

\newcommand*{\lambdaPoly}{\operatorname{\lambda \mathbf{Poly}}}
\newcommand*{\lambdaPolyA}{\operatorname{\lambda Poly_\mathnormal{A}}}
\newcommand*{\lambdaCR}{\operatorname{\lambda \mathbf{CR}}}
\newcommand*{\deltaPoly}{\operatorname{\delta \mathbf{Poly}}}
\newcommand*{\deltaCR}{\operatorname{\delta \mathbf{CR}}}

\newcommand*{\An}{\operatorname{\mathbf{An}}}
\newcommand*{\Sp}{\operatorname{\mathbf{Sp}}}
\newcommand*{\Catinf}{\operatorname{\mathbf{Cat}_\infty}}
\newcommand*{\Prl}{\operatorname{Pr^L}}
\newcommand*{\Prr}{\operatorname{Pr^R}}

\newcommand*{\DZ}{\operatorname{\mathcal{D}(\Z)}}
\newcommand*{\DQ}{\operatorname{\mathcal{D}(\Q)}}

\newcommand*{\CRan}{\operatorname{\mathbf{CR}^{an}}}
\newcommand*{\CRFpan}{\operatorname{\mathbf{CR}_{\mathbf{F}_p}^{an}}}
\newcommand*{\CRFqan}{\operatorname{\mathbf{CR}_{\Fq}^{an}}}
\newcommand*{\CRQan}{\operatorname{\mathbf{CR}^{an}_{\Q}}}
\newcommand*{\CRpan}{\operatorname{\mathbf{CR}^{an}_{(\mathnormal{p})}}}

\newcommand*{\lambdaCRan}{\operatorname{\lambda \mathbf{CR}^{an}}}
\newcommand*{\lambdaCRAan}{\operatorname{\lambda \mathbf{CR}^{an}_\mathnormal{A}}}
\newcommand*{\lambdaCRQan}{\operatorname{\lambda \mathbf{CR}^{an}_{\Q}}}
\newcommand*{\lambdaCRpan}{\operatorname{\lambda \mathbf{CR}^{an}_{(\mathnormal{p})}}}

\newcommand*{\deltaCRan}{\operatorname{\delta \mathbf{CR}^{an}}}
\newcommand*{\deltaCRanWitt}{\operatorname{\delta \mathbf{CR}^{an}_{W_2}}}
\newcommand*{\deltapCRan}{\operatorname{\delta_\mathnormal{p}\mathbf{CR}^{an}}}
\newcommand*{\deltaCRpan}{\operatorname{\delta \mathbf{CR}^{an}_{(\mathnormal{p})}}}

\newcommand*{\C}{\operatorname{\mathcal{C}}}
\newcommand*{\D}{\operatorname{\mathcal{D}}}
\newcommand*{\LEq}{\operatorname{LEq}}
\newcommand*{\End}{\operatorname{End}}

\newcommand*{\forget}{\operatorname{Forget}}
\newcommand*{\free}{\operatorname{Free}}

\newcommand*{\cofreelambda}{\operatorname{Cofree_\lambda}}
\newcommand*{\freelambdap}{\operatorname{Free_\lambda^{(p)}}}
\newcommand*{\forgetlambdap}{\operatorname{Forget_\lambda^{(p)}}}
\newcommand*{\freelambdaQ}{\operatorname{Free_\lambda^\mathbf{Q}}}
\newcommand*{\forgetlambdaQ}{\operatorname{Forget_\lambda^\mathbf{Q}}}
\DeclareMathOperator{\Fr}{Fr_\mathnormal{p}}
\DeclareMathOperator{\Frq}{Fr_\mathnormal{q}}

\newcommand*{\Nadd}{\mathbb{N}_{0}}
\newcommand*{\Nmult}{\mathbb{N}_{>0}}
\newcommand*{\Z}{\mathbf{Z}}
\newcommand*{\Q}{\mathbf{Q}}
\newcommand*{\Fp}{\operatorname{\mathbf{F}}_p}
\newcommand*{\Fq}{\operatorname{\mathbf{F}}_q}

\DeclareMathOperator{\adj}{adj}
\DeclareMathOperator{\an}{an}
\DeclareMathOperator{\Ani}{Ani}
\DeclareMathOperator{\can}{can}
\DeclareMathOperator{\cut}{Cut}
\DeclareMathOperator{\cn}{cn}
\DeclareMathOperator{\cp}{cp}
\DeclareMathOperator{\comp}{comp}
\DeclareMathOperator{\ev}{ev}
\DeclareMathOperator{\fib}{fib}
\DeclareMathOperator{\Fun}{Fun}
\DeclareMathOperator{\id}{id}
\DeclareMathOperator{\op}{op}
\DeclareMathOperator{\pr}{pr}
\DeclareMathOperator{\res}{res}
\DeclareMathOperator{\red}{red}
\DeclareMathOperator{\Map}{Map}
\DeclareMathOperator{\Lan}{Lan}
\DeclareMathOperator{\shift}{shift}

\DeclareMathOperator{\triv}{triv}
\newcommand*{\diffp}{\frac{\operatorname{diff}}{p}}

\DeclareMathOperator{\Coalg}{CoAlg}

\newcommand*{\bigbrace}[1]{\big(#1\big)}
\newcommand\restr[2]{{
  \left.\kern-\nulldelimiterspace
  #1
  \vphantom{|}
  \right|_{#2}
  }}

\begin{document}

\title{Animated $\lambda$-rings and Frobenius lifts}
\author{Edith Hübner}
\address{Mathematisches Institut, Universität Münster, Germany}
\email{ehuebner@uni-muenster.de}
\date{}

\maketitle

\begin{abstract}
	In this note, we study an integral analogue of animated $\delta$-rings: \textit{Animated $\lambda$-rings}. We define animated $\lambda$-rings in terms of animated rings equipped with a structure of coherently compatible
	Frobenius lifts and show that the resulting $\infty$-category is obtained from animating the classical notion of a $\lambda$-ring. Our results build on the theory of animated $\delta$-rings developed by Bhatt and Lurie in the
	context of prismatic cohomology.
\end{abstract}

\tableofcontents

\newpage 
 
\section{Introduction}\label{Sect.: Introduction}

The present paper studies an $\infty$-categorical extension of the classical notions of a $\delta$-ring and $\lambda$-ring: Animated $\delta$-rings and $\lambda$-rings. A theory of animated $\delta$-rings intrinsic to animated rings has been developed by Bhatt and Lurie \cite[Appendix A]{BL} who give the following definition in terms of derived Frobenius lifts:

\begin{definition}[{\cite[A.15]{BL}}]\label{DefIntro: animated delta-rings}
	An animated $\delta$-ring is given by an animated ring $A$ together with an endomorphism $\phi_A \colon A \to A$ and a homotopy $h_A$ which witnesses $\phi_A$ as a lift of the $p$-Frobenius on $A \otimes^L_{\Z} \Fp$.
\end{definition}

Most of the results on animated $\delta$-rings presented in this paper have been shown in \cite{BL} using an equivalent definition based on animated Witt vectors. We stick to \cref{DefIntro: animated delta-rings} which adds a shift in
perspective and complements the exposition in \cite{BL}. Our results on animated $\lambda$-rings are new. We introduce the following definition which should be read as an integral analogue of the definition of an animated $\delta$-ring:

\begin{definition}[Animated $\lambda$-rings]\label{DefIntro.: animated lambda-structure}
	An animated $\lambda$-ring is given by an animated ring $A$ together with the datum of an action $\phi_A \colon B\Nmult \to \CRan$ of the multiplicative monoid $\Nmult$ on $A$ and a family of homotopies $h_{(A,p)}$ witnessing
	\begin{equation*}
		\phi_A(p) \otimes^L \id_{\Fp} \simeq \Fr
	\end{equation*}
	in $\Map_{(\CRFpan)^{B\langle p^c \rangle}}(A \otimes^L_{\Z} \Fp,A \otimes^L_{\Z} \Fp)$ for all primes $p$.
	Here, $\phi_A(p) \colon B\Nadd \to (\CRan)^{B\langle p^c \rangle}$ denotes the map which is classified by $\phi_A$ under the equivalence $B\Nmult \simeq B\Nadd \times B\langle p^c \rangle$ and $A \otimes^L_{\Z} \Fp$
	is endowed with the action $\restr{\phi_A \otimes^L_{\Z} \id_{\Fp}}{B\langle p^c \rangle}$ restricted to the submonoid of natural numbers coprime to $p$.
\end{definition}

We denote the $\infty$-category of animated $\lambda$-rings by $\lambdaCRan$. This definition

\begin{enumerate}
	\item Recovers the classical notion of a $\lambda$-ring precisely when $A$ is a discrete commutative ring, \cref{CorIntro.: discrete animated lambda-rings}.
	\item Gives an intrinsically $\infty$-categorical description of the animation $\Ani(\lambdaCR)$ of the category of ordinary $\lambda$-rings, \cref{ThmIntro.: animation of lambda-rings}.
\end{enumerate}

\subsection*{What do we mean by an intrinsically $\infty$-categorical description?}

There exists a general construction $\C \mapsto \mathcal{P}_\Sigma(\C)$ studied in \cite[5.5.8]{HTT} which assigns to a small $\infty$-category $\C$ with finite
coproducts a presentable $\infty$-category $\mathcal{P}_\Sigma(\C)$ characterised as its free cocompletion under sifted colimits. Setting $\C = \D^{\cp}$ for a $1$-category $\D$ with a sufficient supply of compact projective objects this yields a non-abelian derived $\infty$-category of $\D$, its animation $\Ani(\D)$. In particular, for $\D$ the category of abelian groups the animation recovers the connective part $\DZ_{\geq 0}$ of the unbounded derived $\infty$-category $\DZ$.\\
The question asking for a version of non-abelian derived $\infty$-categories which allow for non-connective objects has been approached by Brantner-Mathew and Bhatt-Mathew in terms of extending filtered monads, see \cite{BrantnerMathew,Raksit}.
Given a sufficiently nice $1$-category $\D$ which is monadic over abelian groups this produces a presentable, non-stable $\infty$-category $\Dalg_{\D}$ which comes with a canonical functor to $\DZ$ and satisfies $(\Dalg_{\D})_{\geq 0} \simeq \Ani(\D)$. This has been used by Raksit to define the $\infty$-category of derived rings \cite{Raksit}.\\

Both of the above constructions provide a widely applicable framework to define higher categorical extensions of algebraic objects. However, owing to their generality, they come with the drawback of leaving objects and mapping spaces rather inexplicit. It is therefore natural to ask for intrinsic descriptions of the such constructed $\infty$-categories in terms of simpler underlying derived objects. For animated $\delta$-rings and $\lambda$-rings this means asking for a description in terms of animated rings.

\subsection*{Let us motivate our definition of an animated $\lambda$-ring:} \Cref{DefIntro.: animated lambda-structure} makes precise the following interpretation of the classical notion of a $\lambda$-ring: A $\lambda$-structure on a
commutative ring $A$ provides the correct definition of the datum of pairwise commuting $p$-Frobenius lifts on $A$ indexed over the positive primes $p$ in the presence of torsion. Formally, our definition of an animated $\lambda$-ring combines the definition of an animated $\delta$-ring (\cref{DefIntro: animated delta-rings}) with a result of Joyal and Rezk (see \cref{Thm.: lambda-rings are global delta-rings} below). \\

The original definition of a $\lambda$-ring due to Grothendieck \cite{GrothChern} is formulated in terms of a family of maps of sets $\lambda^n \colon A \to A$ (for $n \geq 0$) subject to a set of equational axioms which abstract the properties of exterior powers. It was observed by Atiyah and Tall \cite{AtiyahTall} that $\lambda$-rings come with an action of $\Nmult$ defined by a family of pairwise commuting ring endomorphisms $\phi^p \colon A \to A$ (for all primes $p$) each of which lifts the $p$-Frobenius on $A/p$. These commonly go by the name of Adams operations as they recover the cohomology operations on topological K-theory studied by Adams in \cite{Adams}.
The first result pointing in the direction of \cref{DefIntro.: animated lambda-structure} is due to Wilkerson \cite{Wilkerson} who showed that the Adams operations are equivalent to the datum of a $\lambda$-structure whenever the underlying commutative ring is torsion free. This result has been lifted to general rings by Joyal using his notion of a $\delta$-ring. We state his theorem in the slightly simplified 
formulation as proved (much later) by Rezk:

\begin{theorem}[{\cite[Thm 3]{JoyalLambda}} and {\cite[Thm 4.10]{Rezk}}]\label{Thm.: lambda-rings are global delta-rings}
	Let $A$ be a commutative ring. The datum of a $\lambda$-structure on $A$ is equivalent to the datum of ($p$-typical) $\delta$-structures $\delta_p \colon A \to A$ on $A$ for all primes $p$ which satisfy
	$\phi^q \circ \delta_p = \delta_p \circ \phi^q$ for all primes $q$ and $p$ with $q$ unequal to $p$. Here, $\phi^q$ denotes the $q$-Frobenius lift on $A$ associated to $\delta_q$.
\end{theorem}

Recall that the notion of a $\delta$-ring is tailored to capturing the datum of a single $p$-Frobenius lift in the presence of $p$-torsion \cite{JoyalDelta}: A $\delta$-structure on a commutative ring $A$ is given by a map $\delta \colon A \to A$ of sets satisfying the precise set of axioms which ensure that the assignment $x \mapsto \phi(x) = x^p + p\delta(x)$ defines a ring endomorphism on $A$. It is clear by definition that $\phi$ lifts the $p$-Frobenius on $A/p$. It is an observation of Bhatt-Scholze \cite{Prisms} which has been explored in much greater detail by Bhatt-Lurie \cite{BL} that upon passing to animated rings one can reformulate the definition of a $\delta$-ring entirely in terms of derived Frobenius lifts.
Our definition of an animated $\lambda$-structure extends this spirit to the multi-prime case. Here, it is essential not only to require the $\phi_A(p)$ to individually lift the $p$-Frobenius on the derived mod $p$ reduction $A \otimes^L_{\Z} \Fp$,
but to ensure that the remaining action simultaneously lifts to an action by maps of animated ($p$-typical) $\delta$-rings.
This is achieved by requiring the homotopies $h_{(A,p)}$ to live in the category of animated rings equipped with an action of the monoid of natural numbers coprime to $p$.

\subsection*{Why care about animated $\lambda$-rings?}

The notion of a $\delta$-ring is central to the theory of prismatic cohomology of Bhatt and Scholze \cite{Prisms}. Animated $\delta$-rings appear in derived extensions of the theory, most notably in work of Bhatt and Lurie on the Cartier-Witt
stack \cite{APC,BL} and in work of Antieau-Krause-Nikolaus \cite{AKN23}. More recently, derived $\delta$-rings following the extension of monads approach have been defined by Holeman \cite{Holeman}.
Derived $\lambda$-rings and derived binomial rings - a special kind of $\lambda$-ring whose Frobenius lifts are the identity - have been defined by Antieau \cite{Antieau23} and Kubrak, Shuklin and Zakharov \cite{Kubrak23DerivedBinom} following the ideas of Brantner-Mathew and used to model integral homotopy types.\\
There is speculation about $\lambda$-rings appearing in a global version of prismatic cohomology. This is certainly quite vague, but we hope the reader doesn't mind too much. More concretely,
we expect connective algebraic K-theory $K^{\cn}(A) \otimes_{\mathbf{S}} \Z$ of animated rings $A$ to provide examples of animated $\lambda$-rings. We further expect that the Frobenius lift description extends to the non-connective derived setting, see also \cite[Remark 5.14]{Antieau23}.

\subsection*{More details on the contents of this paper}

Let us recall the definition of an animated $\delta$-ring from \cite{BL} in a more diagrammatic fashion:

\begin{definition}[{\cite[A.15]{BL}}]\label{DefIntro.: animated delta-rings}
	An animated $\delta$-ring is given by an animated ring $A$ together with a commutative diagram
	\[\begin{diag}[column sep=2.7em]
		&  |[alias=Z]| A \ar[dr, "\red_A"] & \\[3ex]
		A \ar[ur, "\phi_A"] \ar[rr, "\Fr \circ \red_A"', ""{name=U, above}] \ar[r, from=Z, to=U, phantom, "h_A" {font=\footnotesize, yshift=-0.9ex, xshift=0.5ex}]&& A \otimes^L_{\Z} \Fp
	\end{diag}\]
	in the $\infty$-category of animated rings. We denote the $\infty$-category of animated $\delta$-rings by $\deltaCRan$, see \cref{Def.: deltapCRan via Frobenius lifts}.
\end{definition}

Bhatt and Lurie study animated $\delta$-rings from the vantage point of animated Witt vectors and prove the following theorem which justifies their definition:

\begin{theorem}[{\cite[A.22]{BL}}]\label{ThmIntro.: animation of delta-rings}
	The category $\deltaCRan$ of animated $\delta$-rings is equivalent to the animation $\Ani(\deltaCR)$ of the category of ordinary $\delta$-rings.
\end{theorem}

The essential ingredient going into the proof of their theorem is \cite[A.19]{BL} which states that the free animated $\delta$-ring on a single generator $x$ - defined as the object corepresenting the functor
$\Map_{\CRan}(\Z[X],\forget_\delta(-)) \colon \deltaCRan \to \An$ - coincides with its ordinary counterpart $\delta\{x\} \in \deltaCR$. In the present paper we recover \cite[A.19]{BL} and \cref{ThmIntro.: animation of delta-rings} from the perspective of Frobenius lifts, see \cref{Prop.: free animated delta ring,Thm.: animated delta rings are the animation}. As in \cite{BL}, our proof of \cref{Prop.: free animated delta ring} is based on an analysis of the relevant mapping spaces. Here, we add some generality to \cite{BL} by providing a general formula describing the mapping spaces in the $\infty$-category $\deltaCRan$ in terms of maps between the underlying animated rings (\cref{Prop.: mapping spaces in deltapCRan}). This yields the following description of maps between a pair $(A,\phi_A,h_A)$ and $(B,\phi_B,h_B)$ of animated $\delta$-rings:

\begin{corollary}[{\cref{Rmk.: objects in map deltaCRan}}]
	A map from $(A,\phi_A,h_A)$ to $(B,\phi_B,h_B)$ in $\deltaCRan$ corresponds to the data of
	\begin{enumerate}[itemsep=1.5ex, topsep=2ex]
	\item a map $f\colon A \to B$ between the underlying animated rings
	\item a homotopy $\sigma \colon f \circ \phi_A \simeq \phi_B \circ f$ in $\Map_{\CRan}(A,B)$
	\item a 2-homotopy $\tau \colon {\red_B} \circ \sigma \simeq \sigma(h_A,h_B)$ in the space $\Map_{\CRan}(A,B \otimes^L \Fp)$
		\[\begin{diag}
			({\red_B} \circ {f} \circ \phi_A) \ar[r, bend left=25, ""{name=U, below}, "{\red_B} \circ \sigma"] \ar[r, bend right=25, ""{name=D}, "{\sigma(h_A,h_B)}"'] & ({\red_B} \circ {\phi_B} \circ f)
			\ar[Rightarrow, from=U, to=D, shorten=2mm, "\tau", outer sep=0.8mm]
		\end{diag}\]
	between the mod $p$ reduction of $\sigma$ and the homotopy induced by $h_A$ and $h_B$.
	\end{enumerate}
\end{corollary}

Let us reformulate the definition of an animated $\lambda$-ring in analogy to \cref{DefIntro.: animated delta-rings}:

\begin{definition}[Animated $\lambda$-rings]\label{DefIntro.: animated lambda-structure}
	An animated $\lambda$-ring is given by an animated ring $A$ together with the datum of an action $\phi_A \colon B\Nmult \to \CRan$ of the multiplicative monoid $\Nmult$ on $A$ and homotopies $h_{(A,p)}$
	\[\begin{diag}[column sep=2.7em]
		&  |[alias=Z]| A \ar[dr, "\red_A"] & \\[3ex]
		A \ar[ur, "\phi_A(p)"] \ar[rr, "\Fr \circ \red_A"', ""{name=U, above}] \ar[r, from=Z, to=U, phantom, "h_{(A,p)}" {font=\footnotesize, yshift=-0.9ex, xshift=0.7ex}]&& A \otimes^L_{\Z} \Fp
	\end{diag}\]
	in $\Map_{(\CRan)^{B\langle p^c \rangle}}(A,A\otimes^L_{\Z} \Fp)$ for all primes $p$. Here, $\phi_A(p) \colon \Nadd \to (\CRan)^{B\langle p^c \rangle}$ denotes the map classified by $\phi_A$, see
	\cref{DefIntro.: animated lambda-structure}.
\end{definition}

We denote the $\infty$-category of animated $\lambda$-rings by $\lambdaCRan$ (\cref{Def.: animated lambda-rings}) and prove our main theorem in the second part of the paper:

\begin{theorem}[{\cref{Thm.: lambdaCRan is the animation}}]\label{ThmIntro.: animation of lambda-rings}
	The category $\lambdaCRan$ of animated $\lambda$-rings is equivalent to the animation $\Ani(\lambdaCR)$ of the category of ordinary $\lambda$-rings.
\end{theorem}

As in the case of animated $\delta$-rings, we deduce \cref{ThmIntro.: animation of lambda-rings} from the statement that the free animated $\lambda$-ring on a single generator $x$ is given by its
ordinary counterpart (\cref{Prop.: free animated lambda ring}). Unlike for animated $\delta$-rings, we prove this result using a fracture theorem for animated $\lambda$-rings. This reduces
\cref{Prop.: free animated lambda ring} to the question of determining the free animated $\lambda$-$\Z_{(p)}$-algebra which we show to be discrete. More precisely, it agrees with the $p$-localisation of the free $\lambda$-ring
$\lambda\{x\} \in \lambdaCR$ (\cref{Prop.: free lambda ring free p-local}) which follows from a general description of animated $\lambda$-$\Z_{(p)}$-algebra in terms of animated $\delta$-$\Z_{(p)}$-algebras, see
\cref{Lem.: description animated lambda-Zp-algebras}.

\begin{corollary}[{\cref{Cor.: discrete animated lambda-rings}}]\label{CorIntro.: discrete animated lambda-rings}
	The category $\lambdaCR$ of ordinary $\lambda$-rings is equivalent to the full subcategory $\tau_{\leq 0}(\lambdaCRan)$ of the category of animated $\lambda$-rings whose underlying animated
	ring is discrete.
\end{corollary}

As an easy consequence, we recover Wilkerson's result:

\begin{corollary}[Wilkerson, {\cite[Proposition 1.2]{Wilkerson}}]
	Let $A$ be a torsion-free commutative ring. Then the datum of a $\lambda$-structure on $A$ is equivalent to the datum of pairwise commuting $p$-Frobenius lifts $\phi^p \colon A \to A$ on $A$ for all primes $p$.
\end{corollary}

\begin{proof}
	Combine \cref{CorIntro.: discrete animated lambda-rings} with the fact that $A \otimes^L_{\Z} \Fp \simeq A/p$ for a torsion-free ring $A$.
\end{proof}

\vspace{0.1in}

\subsection{Organisation of the paper}

In \cref{Sect.: Animated delta rings,Sect.: Animated delta rings via Witt vectors,Sect.: The delta map,Sect.: The animation of delta rings} we develop the theory of animated $\delta$-rings from the perspective of Frobenius lifts.
We add some new results to \cite{BL} in \cref{Sect.: Animated delta rings,Sect.: Animated delta rings via Witt vectors,Sect.: The delta map} and provide an alternative proof of \cite[A.19]{BL} on free animated
$\delta$-rings in \cref{Sect.: The animation of delta rings}. The reader familiar with animated $\delta$-rings from \cite{BL} may browse through \cref{Sect.: Animated delta rings via Witt vectors} for a comparison with Bhatt-Lurie and then directly jump ahead to \cref{Sect.: Animated lambda rings}.\\
We formally introduce animated $\lambda$-rings in \cref{Sect.: Animated lambda rings}. In  \cref{Sect.: Fracture square}, we study animated $\lambda$-$\Z_{(p)}$ and $\lambda$-$\Q$-algebras and prove a fracture theorem for animated
$\lambda$-rings. In \cref{Sect.: Free animated lambda rings}, we show that free animated $\lambda$-rings agree with their discrete counterparts and deduce the main theorem of the paper. We conclude with an approach to animated $\lambda$-rings via global animated Witt vectors in \cref{Sect.: Global animated Witt vectors}.

\subsection{Background and Notation}

We use the language of $\infty$-categories as developed in \cite{HTT,HA} throughout the paper and stick to most of the notation presented therein.
Deviating from \cite{HTT}, we follow Clausen and denote the $\infty$-category of spaces (anima) by $\An$. We denote the $\infty$-category of animated rings by $\CRan$ and choose the full subcategory of $\Fun(\Poly^{\op},\An)$ spanned by the functors which preserve finite products as a formal definition, see \cite[Section 25.1.1]{SAG} for a relative version of animated $R$-algebras.
Here, we denote by $\Poly$ the $1$-category of finitely generated polynomial $\Z$-algebras and $\Z$-algebra homomorphisms. Note that $\Poly$ admits finite coproducts which are given by the tensor product. The general theory behind this
construction is detailed in \cite[Section 5.5.8]{HTT}, for the animation of $1$-categories based on Lurie see \cite[Section 5.1.4]{ScholzeCesnavicius}. We content ourselves with listing the most important properties of animated rings and
sketch the proofs insofar as we are unaware of a reference:

\begin{proposition}\label{Prop.: Properties of animated rings} The $\infty$-category $\CRan$ of animated rings has the following properties:
	\begin{enumerate}
	\item The $\infty$-category of animated rings is presentable. Moreover, the Yoneda embedding induces a fully faithful functor $j \colon \Poly \to \CRan$ which preserves finite coproducts.
	\item For any $\infty$-category $\D$ with sifted colimits restriction along $j$ induces an equivalence
		\begin{equation*}
			j^\ast \colon \Fun^\Sigma(\CRan,\D) \xrightarrow{\sim} \Fun(\Poly,\D)
		\end{equation*}
		with inverse given by left Kan extension. Here, $\Fun^\Sigma(\CRan,\D)$ denotes the full subcategory of $\Fun(\CRan,\D)$ spanned by the sifted colimit preserving functors.
	\item \label{Animated rings are sifted colimits} For any $A \in \CRan$ there exists a sifted diagram $j/A \to \Poly$ whose colimit in $\CRan$ is $A$.
	\item \label{Inclusion CR into CRan}There exists a fully faithful functor $\iota \colon \CR \to \CRan$ which preserves small limits and filtered colimits. In particular, it preserves arbitrary coproducts of $\Z[X]$ and admits a left
		adjoint which we denote by $\pi_0$.
	\item \label{Projective objects} Small coproducts of $\Z[X]$ are projective in $\CRan$. Moreover, there exists an equivalence $\CRan \simeq \mathcal{P}_\Delta(\hat{\Poly})$ of $\infty$-categories. Here, the right hand side
		denotes the smallest full subcategory of the presheaf category $\mathcal{P}(\hat{\Poly})$ which contains the essential image of the Yoneda embedding and is closed under geometric realisations. $\hat{\Poly}$
		denotes the category of free $\Z$-algebras and $\Z$-algebra homomorphisms.
	\item \label{Universal property of iota}For any cocomplete $\infty$-category $\D$, the counit of the adjuction
		\begin{equation*}
			\Lan_\iota \colon \Fun(\CR,\D) \to \Fun(\CRan,\D) \colon \iota^\ast
		\end{equation*}
		is an equivalence on the full subcategory $\Fun^{\Delta}(\CRan,\D)$ of functors which preserve geometric realizations.
	\item \label{Coproduct in CRan} The coproduct in the $\infty$-category of animated rings coincides with the functor
		\begin{equation*}
			(-) \otimes^L_{\Z} (-) \colon \CRan \times \CRan \to \CRan
		\end{equation*}
		which is defined as the left Kan extension of the tensor product on commutative rings along $j$ via currying. We also write $\otimes^L$ instead of $\otimes^L_{\Z}$.
	\item \label{Derived mod p reduction} For any prime $p$, we refer to $(-) \otimes^L \Fp \colon \CRan \to \CRan$ as derived mod $p$ reduction. The functor $(-) \otimes^L \Fp$ factors through
		animated $\Fp$-algebras
		\begin{equation*}
			(-) \otimes^L \Fp \colon \CRan \to \CRFpan
		\end{equation*}
		and as such is left adjoint to the canonical functor from animated $\Fp$-algebras to animated rings which forgets the $\Fp$-algebra structure. We denote the unit of this adjunction by $\red$ and the
		homotopy witnessing naturality by $\sigma_{\red}$.
	\end{enumerate}
\end{proposition}

\begin{proof}
	(1) and (2) are \cite[5.5.8.10 and 5.5.8.15]{HTT}. Assertion (3) follows from \cite[5.1.5.3]{HTT} together with the observation that the slice $j/A$ admits finite coproducts for any animated ring $A$. This further uses that sifted colimits in
	$\CRan$ and $\mathcal{P}(\Poly)$ agree. The adjunction in part (4) is induced by the adjunction $\iota \colon \tau_{\leq 0}(\An) \to \An \colon \pi_0$ via postcomposition. This is well-defined since $\pi_0$ preserves small products.
	The stated properties of $\iota \colon \CR \to \CRan$ are inherited from the inclusion of discrete anima into $\An$ since both limits and sifted colimits are computed pointwise in the categories
	$\CR \simeq \Fun^{\prod}(\Poly^{\op},\tau_{\leq 0}(\An))$ and $\CRan$. Finally, $\iota$ preserves arbitrary coproducts of $\Z[X]$ as it preserves filtered colimits and coincides with $j$ on $\Poly$.
	Ad (5): For finite coproducts this is \cite[5.5.8.10]{HTT}. The general case follows from the fact that $\Map_{\CRan} \colon (\CRan)^{\op} \times \CRan \to \An$ factors through $\Sp_{\geq 0}$,
	see \cref{Prop.: underlying anima is a commutative group object}, and that geometric realisations commute with small products in connective spectra. The second part of (5) follows similarly to \cite[5.5.8.22]{HTT}: By
	\cite[5.3.5.9]{HTT} there exists an equivalence $\Fun^{\Delta}(\mathcal{P}_\Delta(\hat{\Poly}),\CRan) \simeq \Fun(\hat{\Poly},\CRan)$ of $\infty$-categories induced by left Kan extension. One argues that the functor
	$\mathcal{P}_\Delta(\hat{\Poly}) \to \CRan$ corresponding to the inclusion $\restr{\iota}{} \colon \hat{\Poly} \to \CRan$ is fully faithful since this holds for $\iota$ and is in fact an equivalence as the elements of
	$\hat{\Poly}$ are projective in $\CRan$ and generate the $\infty$-category of animated rings under geometric realisations, see \cite[5.5.8.13]{HTT}.
	(6) follows from transitivity of left Kan extensions (\cite[7.3.8]{Kerodon}) applied to the inclusions $\hat{\Poly} \subseteq \CR \subseteq \CRan$ together with the equivalence $\Fun^{\Delta}(\CRan,\D) \simeq \Fun(\hat{\Poly},\D)$ of (5).
	To show that $A \otimes^L_{\Z} B$ satisfies the universal property of the  coproduct in animated rings as claimed in (7), one observes that this holds true for $A,B \in \Poly$ and that the functor $(-) \otimes^L_{\Z} (-)$ preserves sifted
	colimits separately in both variables by construction. (8) follows from (7) together with \cite[25.1.4.2]{SAG}.
\end{proof}

\begin{notation}[Animation of a functor]
	We define the animation of an endofunctor $F \colon \CR \to \CR$ as the left Kan extension of the functor $\restr{(\iota \circ F)}{\Poly} \colon \Poly \to \CRan$ along $j$ and denote it by $F^{\an}$. Whenever we have an
	equivalence ${F^{\an}} \circ \iota \simeq \iota \circ F$, i.e. the animation restricts to $F$ on ordinary commutative rings we will denote it again by $F$.
\end{notation}

We assume familiarity of the reader with the classical notions of a $\delta$-ring and a $\lambda$-ring. A modern exposition of the former is given in \cite[Section 2]{Prisms}. We refer the reader to \cite{Yau} for an introduction to $\lambda$-rings. Classical references to the subject matter are \cite{GrothChern,AtiyahTall} for $\lambda$-rings and \cite{JoyalDelta,JoyalLambda} for $\delta$-rings and their connection to $\lambda$-rings. For the global $\delta$-ring perspective on
$\lambda$-rings see also \cite{Rezk}.
We denote the categories of ordinary $\delta$-rings and $\lambda$-rings by $\deltaCR$ and $\lambdaCR$. These are generated under colimits by their full subcategories $\deltaPoly$ and $\lambdaPoly$ of free $\delta$-rings and
$\lambda$-rings on finitely many generators (and retracts of these).

\begin{notation}[Commutative group objects]
	Let $\C$ be an $\infty$-category with finite products. We denote the $\infty$-category of $\mathbf{E}_\infty$-groups in $\C$ by $\CGrp(\C)$. This is an additive $\infty$-category with objects given by functors $X \colon \Fin_\ast \to \C$
	such that the composite $X \circ \cut \colon \Delta^{\op} \to \C$ defines an associative group object in $\C$. %See \cite{xxx}.
\end{notation}

\begin{notation}
	Let $\C$ be an $\infty$-category and $M \in \Mon$ a discrete monoid. We denote by $\C^{BM}$ the $\infty$-category of functors from $BM$ to $\C$ and write $\ev \colon \C^{BM} \to \C$ for the functor which is given by restriction
	along the essentially unique functor $\ast \to BM$.
\end{notation}

We denote the unbounded derived $\infty$-category of abelian groups by $\DZ$ and view it as equipped with its standard t-structure. $\Catinf$ denotes the $\infty$-category of not necessarily small $\infty$-categories.

\subsection*{Acknowledgements}

The author is grateful to Fabian Hebestreit, Thomas Nikolaus and Robert Szafarczyk for mathematical discussions. Special thanks go to Henning Krause for an invitation to Bonn, Jacob Lurie for suggesting fracture squares and Thomas Nikolaus for the generous amount of time he invested in this project. The author was funded by the Deutsche Forschungsgemeinschaft (DFG, German Research Foundation) under Germany's Excellence Strategy EXC 2044 –390685587, Mathematics Münster: Dynamics–Geometry–Structure.

\newpage
\section{Animated $\delta$-rings}\label{Sect.: Animated delta rings}

Fix a prime $p$. In the following sections, we give an exposition of the theory of animated $\delta$-rings from the perspective of Frobenius lifts. Starting point is the definition given at the beginning of \cite[Appendix A]{BL}: An animated $\delta$-ring is an animated ring $A$ together with an endomorphism and a homotopy witnessing it as a lift of the $p$-Frobenius on $A \otimes^L_\Z \Fp$. To make this precise, recall the $p$-Frobenius on animated $\Fp$-algebras:

\begin{definition}[$p$-Frobenius]\label{Def.: animated Frobenius}
	We denote by $\Fr: \CRFpan \to (\CRFpan)^{B\Nadd}$ the animation of the $p$-Frobenius on $\Fp$-algebras. More precisely, it is the left derived functor of the functor which assigns to a finitely generated polynomial $\Fp$-algebra
	the algebra endomorphism given by its $p$-Frobenius.
\end{definition}

\begin{proposition}\label{Prop.: properties Frobenius}
	There exists an equivalence $\ev \circ \Fr \simeq \id_{\CRFpan}$ of endofunctors on $\CRFpan$. In particular, $\Fr$ is conservative and preserves all limits and colimits.
\end{proposition}

\begin{proof}
	All functors involved preserve sifted colimits. It thus suffices to show that $\ev \circ \Fr$ restricts to the identity on $\PolyFp$ which is immediate from
	the definition of the animated $p$-Frobenius. The second assertion follows as both evaluation and the identity satisfy the stated properties.
\end{proof}

\begin{notation}
	In view of \cref{Prop.: properties Frobenius} we denote the value of the animated $p$-Frobenius on $A \in \CRFpan$ by $(A,\Fr \colon A \to A)$. Given any morphism $f \colon A \to B$ of animated
	$\Fp$-algebras, we write $\sigma_{\Fr}(f)$ or in short $\sigma_{\Fr}$ for the homotopy witnessing ${\Fr} \circ f \simeq f \circ \Fr$ in $\Map_{\CRFpan}(A,B)$.
\end{notation}

\begin{definition}[Animated $\delta$-rings]\label{Def.: deltapCRan via Frobenius lifts}
	We define the category $\deltaCRan$ of $p$-typical animated $\delta$-rings as the pullback
	\[\begin{diag}[column sep=3em]
		\deltaCRan \ar[r] \ar[d] \ar[dr, phantom, "\lrcorner", near start]& \CRFpan \ar[d, "\Fr"] \\[4ex]
		(\CRan)^{B\Nadd} \ar[r, "(-) \otimes^L \Fp"] & (\CRFpan)^{B\Nadd}
	\end{diag}\]
	in $\Catinf$.
\end{definition}

\begin{remark}
	In the following, we will refer to objects in $\deltaCRan$ as animated $\delta$-rings. Whenever we need to discern animated $\delta$-rings with respect to different primes $p$ we add the prime
	in question as a subscript and write $\deltapCRan$ instead of $\deltaCRan$.
\end{remark}

\begin{observation}
	Unwinding \cref{Def.: deltapCRan via Frobenius lifts}, we recover the notion of animated $\delta$-ring presented in the introduction: Objects in $\deltaCRan$ are equivalent to triples $(A,\phi_A,h_A)$ consisting of an animated ring $A$
	together with an endomorphism
	$\phi_A\colon A \to A$ in $\CRan$ and a homotopy $h_A$
	\[\begin{diag}[column sep=2.7em]
		&  |[alias=Z]| A \ar[dr, "\red_A"] & \\[3ex]
		A \ar[ur, "\phi_A"] \ar[rr, "\Fr \circ \red_A"', ""{name=U, above}] \ar[r, from=Z, to=U, phantom, "h_A" {font=\footnotesize, yshift=-0.9ex, xshift=0.5ex}]&& A \otimes^L \Fp
	\end{diag}\]
	in $\Map_{\CRan}(A,A\otimes^L \Fp)$ which witnesses $\phi_A$ as a lift of the $p$-Frobenius on $A \otimes^L \Fp$.
\end{observation}

\begin{remark}\label{Rmk.: adjoint homotopies}
	The homotopy $h_A$ appearing in the datum of an animated $\delta$-ring corresponds to a homotopy $\adj(h_A)$ witnessing $\phi_A \otimes^L \id_{\Fp} \simeq \Fr$
	in $\Map_{\CRFpan}(A\otimes^L \Fp,A\otimes^L \Fp)$ by adjunction.
\end{remark}

\begin{observation}\label{Rmk.: properties of fgt to CRan}
	The diagram in \cref{Def.: deltapCRan via Frobenius lifts} is a pullback square in $\Prl$. In particular, the category $\deltaCRan$ of animated $\delta$-rings is presentable and comes with a forgetful functor
	\begin{equation*}
		\forget_\delta \colon \deltaCRan \to (\CRan)^{B\Nadd} \xrightarrow{\ev} \CRan
	\end{equation*}
	to animated rings which is conservative and preserves small colimits. Here, we use that both $\Fr$ and $\ev$ are conservative left adjoint functors. %{\color{gray} Pullbacks of conservative functors are conservative.}
\end{observation}

\begin{construction}\label{Constr.: induced map into fibre product}
	To any pair $(A,\phi_A,h_A)$ and $(B,\phi_B,h_B)$ of animated $\delta$-rings we associate a map
	\begin{equation*}
		\theta(\phi_A,\phi_B) \colon \Map_{\CRan}(A,B) \to \Map_{\CRan}(A,B \times_{B \otimes^L \Fp} B)
	\end{equation*}\\
	which sends $f \colon A \to B$ in $\CRan$ to the map induced by the commutative square
	\[\begin{diag}[column sep=large]
		A \ar[d, "f \circ \phi_A"'] \ar[r, "\phi_B \circ f"] \ar[dr] & B \ar[d, "\red_B"] \\[4ex]
		B \ar[r, "\red_B"] & B \otimes^L \Fp
	\end{diag}\]
	of animated rings with diagonal given by ${\Fr} \circ {\red_B} \circ f$. Commutativity is witnessed by $h_B \circ f$ and the homotopy $\sigma_A(f)$ induced by $h_A$ together with the homotopies $\sigma_{\red}$
	and $\sigma_{\Fr}$. Explicitly, $\sigma_A(f)$ is given by the concatenation of homotopies
	\begin{equation*}
		({\Fr} \circ \sigma_{\red}(f))(\sigma_{\Fr}(f \otimes^L \id_{\Fp}) \circ \red_A)({f\otimes^L \id_{\Fp}} \circ h_A)(\sigma_{\red}(f) \circ \phi_A)
	\end{equation*}
	which witness
		\begin{align*}
		 {\red_B} \circ {f} \circ \phi_A & \simeq {f \otimes^L \id_{\Fp}} \circ {\red_A} \circ \phi_A  \\[0.5ex]
		& \simeq {f \otimes^L \id_{\Fp}} \circ {\Fr} \circ \red_A \\[0.5ex]
		& \simeq {\Fr} \circ {f \otimes^L \id_{\Fp}} \circ \red_A  \\[0.5ex]
		& \simeq {\Fr} \circ {\red_B} \circ f 
	\end{align*}
	functorially in $f$. Formally, $\theta(\phi_A,\phi_B)$ corresponds to the map associated to the commutative square
	\[\begin{diag}[column sep=2em]
		\Map_{\CRan}(A,B) \ar[r, "(\phi_B)_\ast"] \ar[d, "(\phi_A)^\ast"'] \ar[dr, "(\Fr \circ \red_B)_\ast"] & \Map_{\CRan}(A,B) \ar[d, "(\red_B)_\ast"] \\[4ex]
		\Map_{\CRan}(A,B) \ar[r, "(\red_B)_\ast"] & \Map_{\CRan}(A,B \otimes^L \Fp)
	\end{diag}\]
	of anima with homotopies $(h_B)_\ast$ and $\sigma_A$.
\end{construction}

\begin{proposition}[Mapping spaces]\label{Prop.: mapping spaces in deltapCRan}
	The space of maps from $(A,\phi_A,h_A)$ to $(B,\phi_B,h_B)$ in $\deltaCRan$ is given by the pullback
	\[\begin{diag}[column sep=-2ex]
		\Map_{\deltaCRan}((A,\phi_A,h_A),(B,\phi_B,h_B)) \ar[r] \ar[d] \ar[dr, phantom, "\lrcorner", near start]& \Map_{\CRan}(A,B) \ar[d, "\Delta_\ast"] \\[5ex]
		\Map_{\CRan}(A,B) \ar[r, "{\theta(\phi_A,\phi_B)}"] & \Map_{\CRan}(A,B \times_{B \otimes^L \Fp} B)
	\end{diag}\]
	in $\An$. Here, the fibre product $B \times_{B \otimes^L \Fp} B$ is taken over the unit $\red_B \colon B \to B \otimes^L \Fp$
	in $\CRan$ and $\theta(\phi_A,\phi_B)$ denotes the map of \cref{Constr.: induced map into fibre product}.
\end{proposition}

\begin{proof}
	By \cref{Def.: deltapCRan via Frobenius lifts}, the space of maps from $(A,\phi_A,h_A)$ to $(B,\phi_B,h_B)$ sits in a pullback square
	\begin{equation}\label{Diag.: mapping space of animated delta rings}
	\begin{diag}
		\Map_{\deltaCRan}((A,\phi_A,h_A),(B,\phi_B,h_B)) \ar[r] \ar[d] &[-3ex] \Map_{\CRFpan}(A \otimes^L \Fp, B \otimes^L \Fp) \ar[d, "\Fr"] \\[4ex]
		\Map_{(\CRan)^{B\Nadd}}\bigbrace{(A,\phi_A),(B,\phi_B)} \ar[r] & \Map_{(\CRFpan)^{B\Nadd}}\bigbrace{(A \otimes^L \Fp,\Fr),(B \otimes^L \Fp,\Fr)}
	\end{diag}
	\end{equation}
	of anima. Here, the lower horizontal map is induced by $(-)\otimes^L \Fp$ together with the homotopies adjoint to $h_A$ and $h_B$, see \cref{Rmk.: adjoint homotopies}.
	We abbreviate
	\begin{equation*}
		X = \Map_{\CRan}(A,B) \quad \text{and} \quad Y = \Map_{\CRFpan}(A \otimes^L \Fp, B \otimes^L \Fp)
	\end{equation*}
	in the following.	Note that the mapping spaces in the functor categories are given by the pullbacks of the diagrams
	\begin{align*}
		\Map_{\CRan}(A,B) \xrightarrow{(\phi_A^\ast,\phi_{B\ast})} &X \times X \xleftarrow{\;\Delta\;} \Map_{\CRan}(A,B) \\[1.2ex]
		\Map_{\CRFpan}(A \otimes^L \Fp,B \otimes^L \Fp) \xrightarrow{(\Fr^\ast,\Fr_\ast)} &Y \times Y \xleftarrow{\;\Delta\;} \Map_{\CRFpan}(A \otimes^L \Fp,B \otimes^L \Fp).
	\end{align*}
	Inserting this in \cref{Diag.: mapping space of animated delta rings} identifies the mapping space in $\deltaCRan$ with the limit of the diagram
		\begin{equation}\label{Diag.: mapping space deltaCRan expanded}
	\begin{diag}[column sep=3.5em]
		X \ar[d, "{(\phi_A^\ast, \phi_{B \ast})}"'] \ar[r] & Y \ar[d, "{(\Fr^\ast, \,\Fr_\ast)}"'] &[3ex] Y \ar[l, "\sim"'] \ar[d, "\sim"'{anchor=north, rotate=90}] \\[6ex]
		X \times X \ar[r]	& Y \times Y & Y \ar[l, "\Delta \circ \Fr^\ast"'] \\[6ex]
		X \ar[r] \ar[u, "\Delta"] & Y \ar[u, "\Delta"] & Y \ar[u, "\sim"'{anchor=north, rotate=90}] \ar[l, "\Fr^\ast"']
	\end{diag}
	\end{equation}
	of anima. Here, all horizontal maps in the left hand column are induced by $(-) \otimes^L \Fp \colon X \to Y$. We expand the right hand column in \cref{Diag.: mapping space deltaCRan expanded}
	\begin{equation}\label{Diag.: mapping space deltaCRan expanded twice}
	\begin{diag}[column sep=3.5em]
		X \ar[d, "{(\phi_A^\ast, \phi_{B \ast})}"'] \ar[r] & Y \ar[d, "{(\Fr^\ast, \,\Fr_\ast)}"'] &[2ex] Y \ar[l, "\sim"'] \ar[d, "\Fr^\ast"] &[4ex] Y \ar[l, "\sim"'] \ar[d, "\sim"'{anchor=north, rotate=90}] \\[6ex]
		X \times X \ar[r]	& Y \times Y & Y \ar[l, "\Delta"'] & Y \ar[l, "\Fr^\ast"'] \\[6ex]
		X \ar[r] \ar[u, "\Delta"] & Y \ar[u, "\Delta"] & Y \ar[u, "\sim"'{anchor=north, rotate=90}] \ar[l, "\sim"'] & Y \ar[l, "\Fr^\ast"'] \ar[u, "\sim" {anchor=north, rotate=90}]
	\end{diag}
	\end{equation}
	and contract the horizontal pullbacks on the left to obtain the diagram
	\begin{equation}\label{Diag.: mapping space deltaCRan contracted once}
	\begin{diag}[column sep=3.5em]
		X \ar[d, "{\theta'(\phi_A,\phi_B)}"'] \ar[r] & Y \ar[d, "\Fr^\ast"'] &[4ex] Y \ar[l, "\sim"'] \ar[d, "\sim"'{anchor=north, rotate=90}] \\[6ex]
		X \times_Y X \ar[r] & Y & Y \ar[l, "\Fr^\ast"'] \\[6ex]
		X \ar[r] \ar[u, "\Delta"] & Y \ar[u, "\sim" {anchor=north, rotate=90}] & Y \ar[u, "\sim"'{anchor=north, rotate=90}] \ar[l, "\Fr^\ast"']
	\end{diag}
	\end{equation}
	of anima. Here, $\theta'(\phi_A,\phi_B)$ denotes the map induced by the commutative square
	\[\begin{diag}[column sep=5em]
		X \ar[r, "(\phi_B)_\ast"] \ar[d, "(\phi_A)^\ast"'] \ar[dr] & X \ar[d, "(-) \otimes^L \Fp"] \\[6ex]
		X \ar[r, "(-) \otimes^L \Fp"] & Y
	\end{diag}\]
	with diagonal given by ${\Fr_\ast} \circ (-) \otimes^L \Fp$. Commutativity is witnessed by the homotopies adjoint to $h_A$ and $h_B$ as well as $\sigma_{\Fr}$. The claim follows from observing that $\theta'(\phi_A,\phi_B)$ is equivalent to 
	$\theta(\phi_A,\phi_B)$ by adjunction and that the limit of \cref{Diag.: mapping space deltaCRan contracted once} is equivalent to the vertical pullback on its outermost left side.
\end{proof}

\begin{remark}\label{Rmk.: objects in map deltaCRan}
	Unwinding \cref{Prop.: mapping spaces in deltapCRan} shows that maps between the pair of animated $\delta$-rings $(A,\phi_A,h_A)$ to $(B,\phi_B,h_B)$ correspond to triples $(f,\sigma,\tau)$ which consist of
	\begin{enumerate}[itemsep=1.5ex, topsep=2ex]
	\item a map $f\colon A \to B$ between the underlying animated rings
	\item a homotopy $\sigma \colon f \circ \phi_A \simeq \phi_B \circ f$ in $\Map_{\CRan}(A,B)$
	\item a 2-homotopy $\tau \colon {\red_B} \circ \sigma \simeq (h_B \circ f)\sigma_A(f)$ in the space $\Map_{\CRan}(A,B \otimes^L \Fp)$
		\[\begin{diag}
			({\red_B} \circ {f} \circ \phi_A) \ar[r, bend left=25, ""{name=U, below}, "{\red_B} \circ \sigma"] \ar[r, bend right=25, ""{name=D}, "(h_B \circ f)\sigma_A(f)"'] & ({\red_B} \circ {\phi_B} \circ f)
			\ar[Rightarrow, from=U, to=D, shorten=2mm, "\tau", outer sep=0.8mm]
		\end{diag}\]
	between the mod $p$ reduction of $\sigma$ and the homotopy induced by $h_A$ and $h_B$, see \cref{Constr.: induced map into fibre product}.
	\end{enumerate}
\end{remark}

\newpage
\section{The $\delta$-map}\label{Sect.: The delta map}

To any animated ring $A$ one can associate an anima defined by $[A] = A(\Z[X]) \in \An$. We stick to the notation in \cite{BL} and refer to $[A]$ as the underlying anima of $A$. Equivalently, this is given by the space
$\Map_{\CRan}(\Z[X],A)$. It follows from the definition of animated rings that the assignment $A \mapsto [A]$ defines a functor
\begin{equation*}
	[-] \colon \CRan \to \An
\end{equation*}
which is conservative and preserves sifted colimits as well as small limits.

\begin{proposition}\label{Prop.: underlying anima is a commutative group object}
	The functor $[-] \colon \CRan \to \An$ factors through $\ev_{\langle 1 \rangle} \colon \CGrp(\An) \to \An$.
\end{proposition}

\begin{proof}
	The polynomial ring $\Z[X]$ extends to a commutative group object $X$ in $\Poly^{\op}$
	\[\begin{diag}[column sep=small]
		\Fin_\ast \ar[rr, "X"] && \Poly^{\op} \\
		& \ast \ar[ur, "{\Z[X]}"'] \ar[ul, "\langle 1 \rangle"] &
	\end{diag}\]
	which sends a pointed map $\alpha \colon \langle m \rangle \to \langle n \rangle$ to the map $\Z[X_1,\ldots,X_n] \to \Z[Y_1,\ldots,Y_m]$ of $\Z$-algebras defined by $X_i \mapsto \sum_{j \in \alpha^{-1}(i)} Y_j$ for all $1 \leq i \leq n$.
	As animated rings are given by finite product preserving functors from $\Poly^{\op}$ to anima, the assignment $A \mapsto A \circ X$ defines a functor $\CRan \to \CGrp(\An)$ which satisfies
	$(A \circ X)(\langle 1 \rangle) = A(\Z[X])$.
\end{proof}

\begin{remark}\label{Rmk.: Underlying spectrum}
	Alternatively, \cref{Prop.: underlying anima is a commutative group object} follows from a more general factorisation of the underlying anima functor as detailed in the proof of \cite[25.1.2.2]{SAG}: Denote by
	$\theta \colon \CRan \to \CAlg_{\Z}^{\cn}$ the unique sifted colimit preserving functor which restricts on polynomial $\Z$-algebras to the inclusion into $\mathbf{E}_\infty$-algebras over $\Z$.
	Further, let $\phi \colon \CAlg_{\Z}^{\cn} \to \An$ denote the composite
	\begin{equation*}
		\CAlg_{\Z}^{\cn} \simeq \CAlg(\Mod_{\Z}^{\cn}) \to \Mod_{\Z}^{\cn} \simeq \DZ_{\geq 0} \to \Sp^{\cn} \xrightarrow{\Omega^\infty} \An
	\end{equation*}
	where the first two are forgetful functors. Arguing that the functor $\phi \circ \theta$ preserves sifted colimits and agrees with $[-]$ on polynomial algebras over $\Z$ yields an equivalence $[-] \simeq \phi \circ \theta$.
	%Use $\Sp^{\cn} \simeq \CGrp(\An)$, compatible with $\Omega^\infty$ and $\ev_{\langle 1 \range}$.
\end{remark}

\begin{notation}\label{Notation: underlying comm group functor}
	We write $[-]_{\mathbf{E}_\infty} \colon \CRan \to \Sp_{\geq 0} \simeq \CGrp(\An)$ for the functor of \cref{Prop.: underlying anima is a commutative group object}. Since we have an equivalence
	$[-] \simeq \Omega^\infty \circ [-]_{\mathbf{E}_\infty}$ of functors $\CRan \to \An$ it follows that $[-]_{\mathbf{E}_\infty}$ is conservative and preserves small limits as well as sifted colimits.
\end{notation}

\begin{corollary}\label{Cor.: mapping spaces are group objects}
	For any animated ring $A$ and any anima $X$ the space $\Map_{\An}(X,[A])$ lies in the image of the functor $\ev_{\langle 1 \rangle} \colon \CGrp(\An) \to \An$.
\end{corollary}

\begin{proof}
	This follows from observing that $\Map_{\An}(X,[A])$ equals the value of the composite functor
	\begin{equation*}
		\CRan \xrightarrow{[-]} \CGrp(\An) \xrightarrow{\Map_{\An}(X,-)_\ast} \CGrp(\An) \xrightarrow{\ev_{\langle1 \rangle}} \An
	\end{equation*}
	on $A \in \CRan$. Here, we use that $\Map_{\An}(X,-) \colon \An \to \An$ preserves small limits.
\end{proof}

\begin{notation}\label{Not.: nat transformation by endo}
	Any ring endomorphism $f \colon \Z[X] \to \Z[X]$ induces a natural transformation of the functor $[-] \colon \CRan \to \An$ by precomposition. Given any $B \in \CRan$, we will denote the respective
	component of the transformation by
	\begin{equation*}
		[B] \xrightarrow{f(X)_B} [B].
	\end{equation*}
	In the special case of the endomorphism of $\Z[X]$ defined by $X \mapsto pX$ we simply write $[B] \xrightarrow{\cdot p} [B]$. Note that this lifts to the underlying
	chain complex of $B$.
\end{notation}

\begin{proposition}\label{Prop.: fibre sequence on anima}
	For any animated ring $B$ there exists a fibre sequence
	\[\begin{diag}[column sep=3em]
		{[B]} \ar[d] \ar[r, "\cdot p"] & {[B]} \ar[d, "{[\red_B]}"] \\[4ex]
		\ast \ar[r, "0"] & {[B \otimes^L \Fp]}
	\end{diag}\]
	in $\An$ which is natural in $B$. Here, we denote by $0$ the map adjoint to $\Z \xrightarrow{0} B \otimes^L \Fp$ in $\DZ_{\geq 0}$. We write $\sigma_{\fib}^B$ for the null-homotopy witnessing $[\red_B] \circ (\cdot p) \simeq 0$.
\end{proposition}

\begin{proof}
	This follows from \cref{Lemma: mod p fibre sequence} together with the fact that the functor $[-]$ factors through $\DZ_{\geq 0}$ by \cref{Rmk.: Underlying spectrum}.
\end{proof}

\begin{construction}[Dividing through $p$]\label{Constr.: diffp}
	We define a natural transformation
	\begin{equation*}
		\textstyle \diffp \colon [(-) \times_{(-) \otimes^L \Fp} (-)] \to [-]
	\end{equation*}
	of functors $\CRan \to \An$ whose component associated to an animated ring $B$ is given by the limit of the commutative diagram
	\[\begin{diag}
		{[B]} \times {[B]} \ar[d, "{[\red_B] \times [\red_B]}"] \ar[r, "-"] & {[B]} \ar[d, "{[\red_B]}"] \\[3ex]
		{[B \otimes^L \Fp]} \times {[B \otimes^L \Fp]} \ar[r, "-"] & {[B \otimes^L \Fp]} \\[3ex]
		{[B \otimes^L \Fp]} \ar[u, "\Delta"'] \ar[r] & \ast \ar[u, "0"']
	\end{diag}\]
	in $\An$. Here, we identify $[B] \simeq \fib([\red_B])$ functorially in $B$ by \cref{Prop.: fibre sequence on anima} and denote by $-$ the map provided by
	the $\mathbf{E}_\infty$-group structure on $[B]$. Informally, one may think of the map
	\begin{equation*}
		\textstyle (\diffp)_B \colon [B] \times_{[B \otimes^L \Fp]} [B] \to [B]
	\end{equation*}
	as assigning to the element $(x,y,\sigma)$ with $\sigma \colon [\red_B] \circ x \simeq [\red_B] \circ y$ in $\Map_{\An}(\ast,[B \otimes^L \Fp])$ the element $\frac{x-y}{p}$ in the fibre.
\end{construction}

\begin{lemma}\label{Cor.: pb square underlying spaces}
	For any animated ring $B$, we have a pullback square
	\[\begin{diag}
		{[B]} \ar[r, "\Delta"] \ar[d] \ar[dr, phantom, "\lrcorner", near start] &[-1ex] {[B]} \times_{[B \otimes^L \Fp]} {[B]} \ar[d, "(\diffp)_B"] \\[3ex]
		\ast \ar[r, "0"] & {[B]}
	\end{diag}\]
	in $\An$ which is natural in $B$. Here, $\Delta$ is induced by the diagonal.
\end{lemma}

\begin{proof}
	This follows from the definition of $\diffp$ together with the fact that $[B] \xrightarrow{\Delta} [B] \times [B] \xrightarrow{-} [B]$ is a fibre sequence in $\An$.
\end{proof}

\begin{remark}
	The assertions of \cref{Prop.: fibre sequence on anima,Cor.: pb square underlying spaces} hold true on the level of $\DZ_{\geq 0}$. In particular, the natural transformation of \cref{Constr.: diffp} exists as a map of chain complexes.
\end{remark}

\begin{proposition}\label{Prop.: Frobenius homotopy on underlying spaces}
	For any animated ring $B$ there exists a homotopy
	\[\begin{diag}[column sep=2em]
		& |[alias=Z]| {[B]} \ar[dr, "{[\red_B]}"]  & \\[3ex]
		 {[B]} \ar[ur, "X_B^p"] \ar[rr, "{[\Fr \circ \red_B]}"', ""{name=U, above}]  \ar[r, from=Z, to=U, phantom, "\sigma_{\can}^B" {font=\footnotesize, yshift=-0.9ex, xshift=0.5ex}] &&
		 {[B \otimes^L \Fp]}
	\end{diag}\]
	in $\An$ which is functorial in $B$. See \cref{Not.: nat transformation by endo} for the definition of $X_B^p$.
\end{proposition}

\begin{proof}
	${[\red]} \circ X^p$ and $[\Fr \circ \red]$ define natural transformations from $[-]$ to $[(-) \otimes^L \Fp]$ as functors from animated rings to anima.
	By the Yoneda lemma, we have an equivalence
	\begin{equation*}
		\Map_{\Fun(\CRan,\An)}([-],[(-) \otimes^L \Fp]) \simeq \Map_{\CRan}(\Z[X],\Fp[X])
	\end{equation*}
	where the right hand side is discrete. The claim follows from observing that $[\red_{\Z[X]}] \circ X_{\Z[X]}^p$ and $[\Fr \circ \red_{\Z[X]}]$ send $\id_{\Z[X]}$ to the same map
	in $\Map_{\CRan}(\Z[X],\Fp[X])$.
\end{proof}

\begin{construction}[The $\delta$-map]\label{Constr.: delta-map on underlying anima}
	Let $(B,\phi_B,h_B)$ be an animated $\delta$-ring. Observe that $h_B$ defines a homotopy
	\[\begin{diag}[column sep=2em]
		& |[alias=Z]| {[B]} \ar[dr, "{[\red_B]}"]  & \\[3ex]
		 {[B]} \ar[ur, "{[\phi_B]}"] \ar[rr, "{[\Fr \circ \red_B]}"', ""{name=U, above}]  \ar[r, from=Z, to=U, phantom, "{[h_B]}" {font=\footnotesize, yshift=-0.9ex, xshift=0.5ex}] && {[B \otimes^L \Fp]}
	\end{diag}\]
	on underlying anima. Together with the homotopy of \cref{Prop.: Frobenius homotopy on underlying spaces} this gives a null-homotopy
	\begin{equation*}
		[h_B] - \sigma_{\can}^B \colon [\red_B] \circ ([\phi_B] - X_B^p) \simeq 0
	\end{equation*}
	in $\Map_{\An}([B],[B \otimes^L \Fp])$. %This is well-defined by \cref{Cor.: mapping spaces are group objects}.
	By \cref{Prop.: fibre sequence on anima}, we obtain the datum $(\delta_B,h^\delta_B,\tau^\delta_B)$ consisting of
	\begin{enumerate}
		\item a map $\delta_B \colon [B] \to [B]$ of anima
		\item a homotopy $h^\delta_B \colon (\cdot p) \circ \delta_B \simeq [\phi_B] - X_B^p$ in $\Map_{\An}([B],[B])$
		\item a $2$-homotopy $\tau_B^\delta$ in $\Map_{\An}([B],[B \otimes^L \Fp])$ which witnesses the equivalence
				\begin{equation*}
				([h_B] - \sigma_{\can}^B)([\red_B] \circ h_B^\delta) \simeq \sigma_{\fib}^B \circ \delta_B
				\end{equation*}
			of homotopies witnessing $[\red_B] \circ (\cdot p) \circ \delta_B \simeq 0$.
		%\item a $2$-homotopy $\tau_B^\delta$ which witnesses commutativity of the diagram
		%	\[\begin{diag}[column sep=1em]
		%		{[\red_B] \circ (\cdot p) \circ \delta_B} \ar[rr, "{[\red_B] \circ h_B^\delta}", "\sim"']  \ar[dr, "{\sigma_{\fib}^B} \circ \delta_B"', "\sim"] && {[\red_B] \circ ([\phi_B] - X_B^p)} \ar[dl, "{[h_B] - \sigma_{\can}^B}", "\sim"'] & \\[3ex]
		%		& 0 &
		%	\end{diag}\]
		%	in $\Map_{\An}([B],[B \otimes^L \Fp])$.
	\end{enumerate}
\end{construction}

\begin{remark}
	The map $\delta_B$ associated to the animated $\delta$-ring $(B,\phi_B,h_B)$ by \cref{Constr.: delta-map on underlying anima} is equivalent to the composite map
	${(\diffp)_B} \circ \varphi(\phi_B,X_B^p) \colon [B] \to [B]$ of anima. Here, $\varphi$ denotes the map induced by the commutative diagram
	\[\begin{diag}
		{[B]} \ar[r, "{[\phi_B]}"] \ar[d, "X_B^p"'] \ar[dr] &[2ex] {[B]} \ar[d, "{[\red_B]}"] \\[4ex]
		{[B]} \ar[r, "{[\red_B]}"] & {[B \otimes^L \Fp]}
	\end{diag}\]
	in $\An$ filled by the pair of homotopies $[h_B]$ and $\sigma_{\can}^B$, see \cref{Prop.: Frobenius homotopy on underlying spaces}.
\end{remark}

\begin{construction}\label{Constr.: theta on underlying anima}
	To any $(B,\phi_B,h_B) \in \deltaCRan$ and any $(X,\alpha) \in (\An)^{B\Nadd}$ we associate a map
	\begin{equation*}
		\theta(\alpha,\delta_B) \colon \Map_{\An}(X,{[B]}) \to \Map_{\An}(X,[B] \times_{[B \otimes^L \Fp]} [B])
	\end{equation*}
	of anima which sends $g \colon X \to [B]$ in $\An$ to the map induced by the commutative square
	\[\begin{diag}[column sep=3.5em]
		X \ar[r, "{[\phi_B] \circ g}"] \ar[d, "{X_B^p \circ g + p(g \circ \alpha)}"'] \ar[dr] & {[B]} \ar[d, "{[\red_B]}"] \\[5ex]
		{[B]} \ar[r, "{[\red_B]}"] & {[B \otimes^L \Fp]}
	\end{diag}\]
	of anima. Here, the diagonal is given by ${[\Fr \circ \red_B]} \circ g$ and commutativity is witnessed by the homotopies $[h_B] \circ g$ and $\sigma_{\can}^B \circ g + \sigma_{\fib}^B \circ (g \circ \alpha)$.
	More formally, we define $\theta(\alpha,\delta_B)$ as the map associated to the commutative square
	\[\begin{diag}[column sep=2em]
		\Map_{\An}(X,{[B]}) \ar[r,"{[\phi_B]_\ast}"] \ar[d, "(X_B^p)_\ast + (\cdot p)_\ast \circ \alpha^\ast"'] \ar[dr, "{[\Fr \circ \red_B]_\ast}"]& \Map_{\An}(X,{[B]}) \ar[d, "{[\red_B]}_\ast"] \\[5ex]
		\Map_{\An}(X,{[B]}) \ar[r, "{[\red_B]}_\ast" {yshift=0.7ex}] & \Map_{\An}(X,{[B \otimes^L \Fp]})
	\end{diag}\]
	of anima where commutativity is witnessed by $[h_B]_\ast$ and $(\sigma_{\can}^B)_\ast + (\sigma_{\fib}^B)_\ast \circ \alpha^\ast$.
\end{construction}

\begin{lemma}\label{Obs.: difference of deltas}
	There exists an equivalence
	\begin{equation*}
		\textstyle (\diffp)_{\ast} \circ \theta(\alpha,\delta_B) \simeq (\delta_B)_\ast - \alpha^\ast
	\end{equation*}
	of maps $\Map_{\An}(X,[B]) \to \Map_{\An}(X,[B])$ of anima.
\end{lemma}

\begin{proof}
	It follows from \cref{Constr.: diffp} together with functoriality of limits that the left hand side is equivalent to the map induced by the diagram
	\[\begin{diag}[column sep=small]
		\Map_{\An}(X,{[B]}) \ar[r, "\beta"] \ar[d] & \Map_{\An}(X,{[B]}) \ar[d, "{[\red_B]}_\ast"] \\[4ex]
		\Map_{\An}(X,\ast) \ar[r, "0"] & \Map_{\An}(X,{[B \otimes^L \Fp]})
	\end{diag}\]
	in $\An$ with $\beta = [\phi_B]_\ast - ((X_B^p)_\ast + (\cdot p)_\ast \circ \alpha^\ast)$ and commutativity witnessed by $[h_B]_\ast - ((\sigma_{\can}^B)_\ast + (\sigma_{\fib}^B)_\ast \circ \alpha^\ast)$.
	By \cref{Constr.: delta-map on underlying anima}, this diagram is equivalent to the commutative square with $\beta$ replaced by $(\cdot p \delta_B)_\ast - (\cdot p)_\ast \circ \alpha^\ast$ and the null-homotopy
	given by $(\sigma_{\fib}^B)_\ast \circ ((\delta_B)_\ast - \alpha^\ast)$. It is clear that this induces the map $(\delta_B)_\ast - \alpha^\ast$ on the fibre.
\end{proof}

\begin{proposition}\label{Prop.: commutation rule theta maps and underlying anima}
	For any pair $(A,\phi_A,h_A)$ and $(B,\phi_B,h_B)$ of animated $\delta$-rings there exists an equivalence
	\begin{equation*}
		[-] \circ \theta(\phi_A,\phi_B) \simeq \theta(\delta_A,\delta_B) \circ [-]
	\end{equation*}
	of maps $\Map_{\CRan}(A,B) \to \Map_{\An}([A],[B \times_{B \otimes^L \Fp} B])$ of anima.
	Here, $ \theta(\phi_A,\phi_B)$ denotes the map of \cref{Constr.: induced map into fibre product} and $\theta(\delta_A,\delta_B)$ denotes the map of \cref{Constr.: theta on underlying anima} associated to the pair $([A],\delta_A)$.
\end{proposition}

\begin{proof}
	The functor $[-]$ preserves small limits. Thus, the left hand map is induced by the commutative square
	\[\begin{diag}[column sep=1em]
		\Map_{\CRan}(A,B) \ar[r, "(\phi_B)_\ast"] \ar[d, "(\phi_A)^\ast"'] \ar[dr] & \Map_{\CRan}(A,B) \ar[d, "{[-] \circ (\red_B)_\ast}"] \\[6ex]
		\Map_{\CRan}(A,B) \ar[r, "{[-] \circ (\red_B)_\ast}" {yshift=1.5ex}] & \Map_{\An}({[A]},{[B \otimes^L \Fp]})
	\end{diag}\]
	of anima with diagonal $[-] \circ (\Fr \circ \red_B)_\ast$ and homotopies $[-] \circ \sigma_A$ and $[-] \circ (h_B)_\ast$. The right hand map is induced by the diagram
	\[\begin{diag}[column sep=1em]
		\Map_{\CRan}(A,B) \ar[r,"{[\phi_B]_\ast \circ [-]}" {yshift=1ex}] \ar[d, "\substack{(X^p_B)_\ast \circ [-] \\ \\+ {(\cdot p)_\ast} \circ {\delta_A^\ast} \circ [-]}"'] \ar[dr]& \Map_{\An}({[A]},{[B]}) \ar[d, "{[\red_B]}_\ast"] \\[6ex]
		\Map_{\An}({[A]},{[B]}) \ar[r, "{[\red_B]}_\ast" {yshift=1ex}] & \Map_{\An}({[A]},{[B \otimes^L \Fp]})
	\end{diag}\]
	with diagonal $[\Fr \circ \red_B]_\ast \circ [-]$ and homotopies $\xi = (\sigma_{\can}^B)_\ast \circ [-] + (\sigma_{\fib}^B)_\ast \circ \delta_A^\ast \circ [-]$ and $[h_B]_\ast \circ [-]$. It thus suffices to produce a homotopy
	\begin{equation*}
		\nu \colon [\phi_A]^\ast \circ [-] \simeq ((x_B^p)_\ast + (\cdot p)_\ast \circ \delta_A^\ast) \circ [-]
	\end{equation*}
	together with a $2$-homotopy $\tau \colon \xi([\red_B]_\ast \circ \nu) \simeq [\sigma_A] \circ [-]$. We define $\nu$ as the composite
	\begin{equation*}
		\nu = (\delta_A^\ast \circ \sigma_{\cdot p} + (X_B^p)_\ast \circ [-])((h^\delta_A)^\ast \circ [-] + \sigma_{X^p})
	\end{equation*}
	of homotopies witnessing
	\begingroup
	\addtolength{\jot}{1ex}
	\begin{align*}
		[\phi_A]^\ast \circ [-] &\simeq ([\phi_A] - X_A^p)^\ast \circ [-] + (X_A^p)^\ast \circ [-]\\
		&\simeq ((\cdot p) \circ \delta_A)^\ast \circ [-] + (X_B^p)_\ast \circ [-] \\
		&\simeq \delta_A^\ast \circ (\cdot p)_\ast \circ [-] + (X_B^p)_\ast \circ [-]
	\end{align*}
	\endgroup
	with $\sigma_{\cdot p}$ and $\sigma_{X^p}$ the homotopies witnessing naturality of $(\cdot p)$ and $X^p$. Existence of the 
	$2$-homotopy then follows from \cref{Prop.: fibre sequence on anima,Prop.: Frobenius homotopy on underlying spaces} together with the $2$-homotopy $\tau_A^\delta$ of \cref{Constr.: delta-map on underlying anima}.
\end{proof}

\newpage
\section{Animated $\delta$-rings via $p$-typical Witt vectors}\label{Sect.: Animated delta rings via Witt vectors}

$\delta$-structures on animated rings admit an equivalent definition in terms of animated Witt vectors. This relies on a remark in \cite[Remark 2.5]{Prisms} which states that $p$-typical length-$2$
Witt vectors $W_2(A)$ of an animated ring $A$ can be identified with a fibre product over $A \otimes^L_{\Z} \Fp$ in the category of animated rings. Bhatt and Lurie revisit the pullback square in detail in \cite[A.10]{BL} and formally set
up their study of animated $\delta$-rings in terms of animated Witt vectors. In this section, we recall the perspective on animated $\delta$-rings from \cite{BL} and prove that both definitions give rise to equivalent $\infty$-categories (\cref{Prop.: equivalence of definitions of delta rings via Witt and Frobenius}).

\begin{definition}[{\cite[A.11]{BL}}]
	An animated $\delta$-ring is an animated ring $A$ together with a commutative diagram
	\[\begin{diag}[column sep=2.7em]
		&  |[alias=Z]| W_2(A) \ar[dr, "w_0"] & \\[3ex]
		A \ar[ur, "s_A"] \ar[rr, "\id_A"', ""{name=U, above}] \ar[r, from=Z, to=U, phantom, "h_A" {font=\footnotesize, yshift=-0.9ex, xshift=0.1ex}]&& A
	\end{diag}\]
	in the category of animated rings.
\end{definition}
Here, $W_2$ denotes animated $2$-truncated $p$-typical Witt vectors and $w_0$ the zeroth animated ghost component, see \cite[A.1 and A.9]{BL} and \cite{Hesselholt} for details on the non-derived version.

\begin{remark}\label{Rmk.: animated delta rings via Witt vectors}
	We will denote the $\infty$-category of animated $\delta$-rings as defined in \cite{BL} by $\deltaCRanWitt$. This admits a formal definition in terms of the following pullback
	\begin{equation}\label{Eq.: dCRanWitt}
		\begin{diag}[column sep=3em]
			\deltaCRanWitt \ar[r] \ar[d] \ar[dr, phantom, "\lrcorner" near start] & \LEq(\id,W_2) \ar[d, "(w_0)_\ast"]\\[4ex]
			\CRan \ar[r, "\id"] & (\CRan)^{B\Nadd}
		\end{diag}
	\end{equation}
	in $\Catinf$. Here, $\LEq(\id,W_2)$ denotes the $\infty$-category of pairs $(A,s_A)$ which consist of an animated ring $A$ together with a morphism $s_A \colon A \to W_2(A)$ of animated rings. See
	\cite[II.1.4]{NikolausTC} for a general definition of lax equalizers. Note that the diagram defining $\deltaCRanWitt$ is in fact both a pullback square in $\Prl$ and $\Prr$ and that all functors in \cref{Eq.: dCRanWitt}
	are conservative.
\end{remark}

We recall some properties of animated truncated $p$-typical Witt vectors as are shown in \cite{BL}:

\begin{proposition}[{\cite[A.3]{BL}}]\label{Prop.: underlying anima of Witt}
	For any $n \geq 1$ and any animated ring $A$, the collection $\comp_m$ of animated Witt components ($1 \leq m \leq n$) induces an equivalence
	$[W_n(A)] \simeq [A]^n$ of underlying anima which is natural in $A$.
\end{proposition}

\begin{proposition}[{\cite[A.4]{BL}}]\label{Prop.: animated Witt vectors on CR}
	For any $n \geq 1$, animated $n$-truncated $p$-typical Witt vectors restrict to ordinary $p$-typical Witt vectors on commutative rings, i.e. there exists a natural equivalence
	$W_n \circ \iota \simeq \iota \circ W_n$ of functors $\CR \to \CRan$.
\end{proposition}

\begin{corollary}\label{Cor.: adjointable square Witt vectors}
	For any $n \geq 1$, the diagram
	\[\begin{diag}
		\CR \ar[r, "W_n"] \ar[d, "\iota"] & \CR \ar[d, "\iota"] \\[4ex]
		\CRan \ar[r, "W_n"] & \CRan
	\end{diag}\]
	in $\Catinf$ is vertically left adjointable.
\end{corollary}

\begin{proof}
	The left adjoint to the inclusion is given by $\pi_0$. We wish to show that the mate transformation
	\begin{equation*}
		\pi_0 W_n \xrightarrow{\pi_0 W_n \eta} \pi_0 W_n \iota \pi_0 \xrightarrow{\sim} \pi_0 \iota W_n \pi_0 \xrightarrow{\epsilon_{W_n \pi_0}} W_n \pi_0
	\end{equation*}
	is an equivalence. As the counit $\epsilon$ is an equivalence, it suffices to show that the unit $\eta$ induces an equivalence. This can be 
	checked after passing to underlying anima where \cref{Prop.: underlying anima of Witt} gives an equivalence $[\pi_0 W_n(\eta)] \simeq [\pi_0(\eta)]^n$. The claim follows as $\pi_0 (\eta)$ is an equivalence.
\end{proof}

\begin{corollary}\label{Cor.: fully faithful embedding of deltaCR}
	The inclusion $\iota \colon \CR \to \CRan$ induces a fully faithful functor $\deltaCR \to \deltaCRanWitt$ which preserves finite coproducts when restricted to $\deltaPoly$.
\end{corollary}

\begin{proof}
	Replacing animated with ordinary commutative rings in \cref{Rmk.: animated delta rings via Witt vectors} gives a definition of ordinary $\delta$-rings, see e.g. \cite[Remark 2.4]{Prisms}.
	\Cref{Prop.: animated Witt vectors on CR} implies that the inclusion induces a fully faithful functor on lax equalizers which sends $(A,s_A)$ to $(\iota(A),\iota(A) \to \iota W_2(A) \simeq W_2(\iota A))$. This extends to a natural transformation
	of functors $\Lambda^2_2 \to \Catinf$
	\begin{equation}\label{Eq.: nat transformation inclusion deltaCR}
	\begin{diag}
		\LEq(\id_{\CR},W_2) \ar[r, "(w_0)_\ast"] \ar[d, "\text{ind}(\iota)"] & (\CR)^{B\Nadd} \ar[d, "\iota_\ast"]& \CR \ar[l, "\id"'] \ar[d, "\iota"]\\[4ex]
		\LEq(\id_{\CRan},W_2) \ar[r, "(w_0)_\ast"] & (\CRan)^{B\Nadd} & \CRan \ar[l, "\id"']
	\end{diag}
	\end{equation}
	which induces a fully faithful functor $\iota \colon \deltaCR \to \deltaCRanWitt$ on the limit. It suffices to prove the second assertion for the composite ${\forget_\delta} \circ \iota \colon \deltaCR \to \CRan$.
	Observe that ${\forget_\delta} \circ \iota \simeq \iota \circ \forget_\delta$ by construction where the right hand side forgets the ordinary $\delta$-structure and then includes into animated rings. The claim now follows as the underlying
	commutative ring of $A \in \deltaPoly$ is of the form $\Z[x_i |\, i \in I]$ for $I$ countably infinite and the inclusion $\iota \colon \CR \to \CRan$ preserves small coproducts of $\Z[x]$.
\end{proof}

\begin{proposition}\label{Prop.: pi0 is left adjoint to inclusion of deltaCR}
	The fully faithful functor $\deltaCR \to \deltaCRanWitt$ of \cref{Cor.: fully faithful embedding of deltaCR} admits a left adjoint which is given by the assignment $A \mapsto \pi_0(A)$ on underlying
	animated rings.
\end{proposition}

\begin{proof}
	The natural transformation defining the inclusion of \cref{Cor.: fully faithful embedding of deltaCR} admits pointwise left adjoints all of which are induced by $\pi_0 \colon \CRan \to \CR$ and satisfies
	the mate condition. This follows essentially from \cref{Cor.: adjointable square Witt vectors}. The claim thus follows from the description of adjoints in functor $(\infty,2)$-categories \cite[Corollary 5.2.9]{Rune24} together with the fact that
	$\lim$ preserves adjunctions as a functor of $(\infty,2)$-categories.
\end{proof}

\begin{notation}
	We denote the left adjoint of \cref{Prop.: pi0 is left adjoint to inclusion of deltaCR} again by $\pi_0 \colon \deltaCRanWitt \to \deltaCR$.
\end{notation}

\begin{lemma}\label{Lemma: mod p fibre sequence}
	For any $A \in \DZ$ there exists a natural fibre sequence
	\[\begin{diag}
		A \ar[r, "\cdot p"] & A \ar[r, "\red_A"] & A \otimes^L_\Z \Fp
	\end{diag}\]
 	in $\DZ$. This is a fibre sequence in $\DZ_{\geq 0}$ whenever $A$ is connective.
\end{lemma}

\begin{proof}
	This follows from tensoring the fibre sequence $\Z \xrightarrow{\cdot p} \Z \to \Fp$ in $\DZ$ with $A$. Here we use that $(-) \otimes^L_\Z A \colon \DZ \to \DZ$ is exact as well as right t-exact whenever $A$ is connective.
\end{proof}

We end this section with showing that the $\infty$-categories of animated $\delta$-rings defined in terms of Frobenius lifts and animated Witt vectors are equivalent. To this end, we give a full proof of the result of Bhatt-Lurie-Scholze:

\begin{proposition}[{\cite[A.10]{BL} and \cite[Remark 2.5]{Prisms}}]\label{Prop.: pullback square for animated delta structures}
	For any animated ring $A$ there exists a pullback square
	\begin{equation}\label{Eq.: pullback square Witt}
	\begin{diag}[column sep=3em]
		W_2(A) \ar[d, "w_0"'] \ar[r, "w_1"] \ar[dr, phantom, "\lrcorner", near start] & A \ar[d, "\red_A"] \\[5ex]
		A \ar[r, "\Fr \circ \red_A"] & A \otimes^L \Fp
	\end{diag}
	\end{equation}
	in $\CRan$ which is functorial in $A$. Here, $w_0$ and $w_1$ denote animated ghost components.
\end{proposition}

\begin{proof}
	Denote by $F^{\an}$ the animation of the functor $\CR \to \Fun(\Delta^1 \times \Delta^1,\CR)$ which assigns to $A \in \CR$ the commutative square $F(A)$
	\[\begin{diag}[column sep=3em]
			W_2(A) \ar[d, "w_0"'] \ar[r, "w_1"] & A \ar[d, "\red_A"] \\[5ex]
			A \ar[r, "\Fr \circ \red_A"] & A \otimes_{\Z} \Fp
	\end{diag}\]
	in $\CR$. Left adjoints preserve left Kan extensions, thus $F^{\an}(A)$ agrees with \cref{Eq.: pullback square Witt} for any animated ring $A$. Note that $F^{\an}(A)$ is a pullback square in animated
	rings for any $A \in \Poly$ as polynomial $\Z$-algebras are $p$-torsion free. We show that the collection of animated rings $A \in \CRan$ for which $F^{\an}(A)$ is a pullback square is closed
	under sifted colimits in $\CRan$. For this, observe that \cref{Eq.: pullback square Witt} is a pullback in $\CRan$ if and only if it is a pullback underlying in $\DZ$. This uses
	\begin{enumerate}[itemsep=1.5ex, topsep=2ex]
	\item The forgetful functor from animated rings to $\DZ_{\geq 0}$ is conservative and preserves small limits.
	\item The limit of $\restr{F^{\an}(A)}{\Lambda^2_2}$ in $\DZ$ is connective for any animated ring $A$. This can be seen from \cref{Lemma: mod p fibre sequence} together with the pasting law for pullbacks.
	\end{enumerate}
	Closure under sifted colimits now follows from stability of $\DZ$ since $F^{\an}$ and the forgetful functor $\CRan \to \DZ$ preserve sifted colimits.
\end{proof}

\begin{proposition}\label{Prop.: equivalence of definitions of delta rings via Witt and Frobenius}
	There exists an equivalence of $\infty$-categories $\deltaCRan \simeq \deltaCRanWitt$.
\end{proposition}

\begin{proof}
	\Cref{Prop.: pullback square for animated delta structures} immediately shows that the $\infty$-category $\LEq(\id,W_2)$ admits a description in terms of
	a pullback square
	\[\begin{diag}[column sep=small]
		\LEq(\id,W_2) \ar[r, "(w_1)_\ast"] \ar[d, "(w_0)_\ast"] \ar[dr, phantom, "\lrcorner" near start]& (\CRan)^{B\Nadd} \ar[d, "\ev \times (-) \otimes^L \Fp"]  \\[4ex]
		(\CRan)^{B\Nadd} \ar[r, "\alpha"] & \CRan \times_{\CRFpan} (\CRFpan)^{B\Nadd}
	\end{diag}\]
	in $\Catinf$. Here, the fibre product in the lower right hand corner is taken over $(-) \otimes^L \Fp$ and $\ev$ and $\alpha$ is induced by the pair of functors $(\ev,{\Fr} \circ (-) \otimes^L \Fp)$. Pasting with the pullback
	square in \cref{Rmk.: animated delta rings via Witt vectors} gives the following pullback
	\begin{equation}\label{Eq.: pb square for deltaCRanWitt}
	\begin{diag}[column sep=small]
		\deltaCRanWitt \ar[r] \ar[d] \ar[dr, phantom, "\lrcorner" near start] & (\CRan)^{B\Nadd} \ar[d, "\ev \times (-) \otimes^L \Fp"]  \\[4ex]
		\CRan \ar[r, "\beta"] &  \CRan \times_{\CRFpan} (\CRFpan)^{B\Nadd}
	\end{diag}
	\end{equation}
	in $\Catinf$. The lower horizontal functor in \cref{Eq.: pb square for deltaCRanWitt} maps $A \in \CRan$ to the object $(A,(A \otimes^L \Fp,\Fr))$ in the fibre product. Further observe that each square in the diagram below
	\[\begin{diag}[column sep=small]
		\CRan \ar[r, "\beta"] \ar[d, "(-)\otimes^L \Fp"] & \CRan \times_{\CRFpan} (\CRFpan)^{B\Nadd} \ar[r, "\pr"] \ar[d, "\pr"] & \CRan \ar[d, "(-) \otimes^L\Fp"] \\[3ex]
		\CRFpan \ar[r, "\Fr"] & (\CRFpan)^{B\Nadd} \ar[r, "\ev"] & \CRFpan
	\end{diag}\]
	is a pullback square: The right hand square is such by definition and the assertion for the outer square holds since ${\pr} \circ \beta \simeq \id$ and $\ev \circ \Fr \simeq \id$ by \cref{Prop.: properties Frobenius}.
	Pasting the left hand square and \cref{Eq.: pb square for deltaCRanWitt} thus yields the equivalence $\deltaCRan \simeq \deltaCRanWitt$.
\end{proof}

\newpage
\section{Free animated $\delta$-rings}\label{Sect.: The animation of delta rings}

Bhatt and Lurie show in \cite{BL} that the $\infty$-category of animated $\delta$-rings is equivalent to the animation of the $1$-category of $\delta$-rings. Their result rests crucially on identifying free
animated $\delta$-rings with their ordinary counterparts. In this section, we give an alternative proof of this result from the perspective of Frobenius lifts.

\begin{proposition}[Free $\delta$-rings]
	The free $\delta$-ring on a single generator $x$ is given by the polynomial ring $\Z[x_n |\, n\geq 0]$ equipped with the unique $\delta$-structure determined by the Frobenius lift $\phi$ defined by
	$\phi(x_n) = x_n^p + p x_{n+1}$ for all $n \geq 0$.  This identifies with the animated $\delta$-ring
	\[\begin{diag}[column sep=-0.5em]
		&  |[alias=Z]| \Z[x_n |\, n\geq 0] \ar[dr, "\red"] & \\[3ex]
		\Z[x_n |\, n\geq 0] \ar[ur, "\phi"] \ar[rr, "\Fr \circ \red"', ""{name=U, above}]  \ar[r, from=Z, to=U, phantom, "" {font=\footnotesize, yshift=-0.8ex, xshift=0.3ex}] && \Z[x_n |\, n\geq 0] \otimes^L \Fp
	\end{diag}\]
	which we denote by $\delta\{x\}$. Here, we drop the homotopy from notation as $\Z[x_n |\, n\geq 0] \otimes^L \Fp \simeq \Fp[x_n |\, n\geq 0]$ is discrete.
\end{proposition}

\begin{proof}
	See e.g. \cite[Théorème 1]{JoyalDelta} or \cite[Lemma 2.11]{Prisms}.
\end{proof}

\begin{proposition}[Free animated $\delta$-rings, {\cite[A.19.]{BL}}]\label{Prop.: free animated delta ring}
	The map $\eta \colon \Z[X] \to \forget_\delta (\delta\{x\}) \simeq \Z[x_n |\, n\geq 0]$ in $\CRan$ defined by $x \mapsto x_0$ exhibits $\delta\{x\}$ as the free animated $\delta$-ring.
	That is, $\eta$ induces an equivalence
	\begin{equation*}
		{\eta^\ast} \circ \forget_\delta \colon \Map_{\deltaCRan}(\delta\{x\},(B,\phi_B,h_B)) \xrightarrow{\sim} \Map_{\CRan}(\Z[X], B)
	\end{equation*}
	for any animated $\delta$-ring $(B,\phi_B,h_B)$.
\end{proposition}

\begin{proof}
	Fix $(B,\phi_B,h_B)$. By \cref{Prop.: mapping spaces in deltapCRan} the space of maps from $\delta\{x\}$ to $(B,\phi_B,h_B)$ in $\deltaCRan$ sits in the
	pullback square
	\[\begin{diag}[column sep=0.2em]
		\Map_{\deltaCRan}((\delta\{x\},\phi, \sigma_{\triv}),(B,\phi_B,h_B)) \ar[d] \ar[r] \ar[dr, phantom, "\lrcorner", near start] & \Map_{\CRan}(\Z[x_n | \,n \geq 0],B) \ar[d, "\Delta_\ast"] \\[4ex]
		\Map_{\CRan}(\Z[x_n | \,n \geq 0],B) \ar[r, "{\theta(\phi,\phi_B)}"] & \Map_{\CRan}(\Z[x_n | \,n \geq 0],B \times_{B \otimes^L \Fp} B) 
	\end{diag}\]
	of anima. Here the unlabeled arrows are induced by the forgetful functor from animated $\delta$-rings to animated rings.
	Observe that the functor $[-] \colon \CRan \to \An$ admits a left adjoint which sends $\ast \in \An$ to $\Z[X]$. Thus, by adjunction, there exist equivalences
	\begin{equation*}%\label{Diag.: adjoint diagram free delta ring}
	\begin{diag}
		\Map_{\CRan}(X,B) \ar[d, "\sim" {anchor=north, rotate=90}] \ar[r, "{\theta(\phi,\phi_B)}"] & \Map_{\CRan}(X,B \times_{B \otimes^L \Fp} B) \ar[d, "\sim" {anchor=north, rotate=90}] &
		\Map_{\CRan}(X,B) \ar[d, "\sim" {anchor=north, rotate=90}] \ar[l, "\Delta_\ast"'] \\[4ex]
		\Map_{\An}(\coprod_{n \geq 0} \ast, [B]) \ar[r, "\overline{\theta}"] & \Map_{\An}(\coprod_{n \geq 0} \ast, [B] \times_{[B \otimes^L \Fp]} [B]) & \Map_{\An}(\coprod_{n \geq 0} \ast, [B]) \ar[l, "\Delta_\ast"']
	\end{diag}\end{equation*}
	with $X = \Z[x_n | \, n \geq 0]$ in $\CRan$. The vertical equivalences are given by the composite $\gamma = u^\ast \circ [-]$ where we denote by $u \colon \coprod_{n \geq 0} \ast \to [X]$ the respective component of the unit.\\

	Let $\shift$ denote the endomorphism of the coproduct $\coprod_{n \geq 0} \ast$ in $\An$ which is defined by ${\shift} \circ \iota_{n} \simeq \iota_{n+1}$ for all $n \geq 0$. We claim that 
	$\overline{\theta}$ is equivalent to the map $\theta(\shift,\delta_B)$ of \cref{Constr.: theta on underlying anima}.	To this end, observe that we have equivalences
	\begin{align*}
		\overline{\theta} &\simeq \gamma \circ \theta(\phi,\phi_B) \circ \gamma^{-1} \\[1.2ex]
		&\simeq u^\ast \circ [-] \circ \theta(\phi,\phi_B) \circ \gamma^{-1} \\[1.2ex]
		& \simeq u^\ast \circ \theta(\delta,\delta_B) \circ [-] \circ \gamma^{-1}
	\end{align*}
	by \cref{Prop.: commutation rule theta maps and underlying anima} where $\delta \colon [X] \to [X]$ denotes the $\delta$-structure map on $\delta\{x\}$. It thus suffices to show that
	\begin{equation*}
		u^\ast \circ \theta(\delta,\delta_B) \simeq \theta(\shift,\delta_B) \circ u^\ast.
	\end{equation*}
	This follows from observing that the diagram of anima
	\[\begin{diag}[column sep=4em]
		\coprod_{n \geq 0} \ast \ar[r, "u"] \ar[d, "\shift"'] & {[X]} \ar[d, "\delta"] \\[4ex]
		\coprod_{n \geq 0} \ast \ar[r, "u"] & {[X]} 
	\end{diag}\]
	commutes. \Cref{Cor.: pb square underlying spaces} exhibits $\Delta_\ast$ as the fibre of $(\diffp)_\ast$:
	\begin{equation}\label{Diag.: pb for simplifying free delta}
	\begin{diag}
		& \Map_{\An}(\coprod_{n \geq 0} \ast, [B]) \ar[r] \ar[d, "\Delta_\ast"] \ar[dr, phantom, "\lrcorner", near start]& \ast \ar[d, "0"] \\[4ex]
		\Map_{\An}(\coprod_{n \geq 0} \ast, [B]) \ar[r, "\overline{\theta}"] & \Map_{\An}(\coprod_{n \geq 0} \ast, [B] \times_{[B \otimes^L \Fp]} [B]) \ar[r, "{(\diffp)_\ast}"] & \Map_{\An}(\coprod_{n \geq 0} \ast, [B])
	\end{diag}\end{equation}
	\Cref{Obs.: difference of deltas} together with the equivalence $\overline{\theta} \simeq \theta(\shift,\delta_B)$ shows that the lower horizontal composite in \cref{Diag.: pb for simplifying free delta} is equivalent to the map
 	\begin{equation*}
		(\delta_B)_\ast - \shift^\ast \colon \Map_{\An}(\textstyle\coprod_{n \geq 0} \ast, [B]) \to \Map_{\An}(\coprod_{n \geq 0} \ast, [B]).
	\end{equation*}
	This identifies the space of maps from $\delta\{x\}$ to $(B,\phi_B,h_B)$ in $\deltaCRan$ with the fibre
	\[\begin{diag}[column sep=2em]
		\Map_{\deltaCRan}(\delta\{x\},B) \ar[d] \ar[r] \ar[dr, phantom, "\lrcorner", near start] & \ast \ar[d, "0"] \\[5ex]
		\prod_{n \geq 0} {[B]} \ar[r, "{(\delta_B)^{\Nadd} - \shift}"] & \prod_{n \geq 0} {[B]}
	\end{diag}\]
	of the map $f= \prod_{n}\delta_B - \shift$ of anima. Further note that the map ${\eta^\ast} \circ {\forget_\delta}$ in the proposition translates to the composite
	\begin{equation*}
		\fib(f) \to \textstyle \prod_{n \geq 0} [B] \xrightarrow{\pr_0} [B].
	\end{equation*}
	We claim that the commutative diagram
	\begin{equation}\label{Diag.: fibre of deltaB - shift}
		\begin{diag}[column sep=2.5em]
		{[B]} \ar[r] \ar[d, "(\delta_B^n)_{n \geq 0}"'] & \ast \ar[d, "0"] \\[4ex]
		\prod_{n \geq 0} {[B]} \ar[r, "f"] & \prod_{n \geq 0} {[B]}
	\end{diag}
	\end{equation}
	identifies $[B] \simeq \fib(f)$ in $\An$. To see that \cref{Diag.: fibre of deltaB - shift} is a pullback square consider the extended diagram
		\[\begin{diag}[column sep=2.5em]
		{[B]} \ar[r] \ar[d, "(\delta_B^n)_{n \geq 0}"'] & \ast \ar[d, "0"] \\[5ex]
		\prod_{n \geq 0} {[B]} \ar[r, "f"] \ar[d, "\pr_0"'] & \prod_{n \geq 0} {[B]} \ar[d] \\[5ex]
		{[B]} \ar[r] & \ast
	\end{diag}\]
	of anima. The composite ${\pr_0} \circ (\delta_B^n)_{n \geq 0}$ is equivalent to $\id_{[B]}$. Thus, by the pasting law it suffices to show that the lower square is a
	pullback diagram. This follows from observing that the induced map $(\pr_0,f) \colon \prod_{n \geq 0} {[B]} \to {[B]} \times \prod_{n \geq 0} {[B]}$
	is an equivalence with inverse given by the map
	\begin{equation*}
		\textstyle {[B]} \times \prod_{n \geq 0} {[B]} \simeq \prod_{n \geq -1} {[B]} \xrightarrow{(g_n)_{n \geq 0}} \prod_{n \geq 0} {[B]}
	\end{equation*}
	inductively defined by $g_0 = \pr_{-1}$ and $g_n = \delta_B \circ g_{n-1} - \pr_{n-1}$ for all $n \geq 1$.
\end{proof}

\begin{corollary}\label{Cor.: Free delta functor}\cite[A.20]{BL}.
	\begin{enumerate}
	\item The forgetful functor $\forget_\delta \colon \deltaCRan \to \CRan$ admits a left adjoint $\free_\delta$.
	\item The category $\deltaCRan$ is equivalent to the category of algebras over the monad $\forget_\delta \circ \free_\delta$ on $\CRan$.
	\end{enumerate}
\end{corollary}

\begin{proof}
	(1) follows from \cref{Prop.: free animated delta ring} together with the observation that $\Z[X]$ generates $\CRan$ under colimits. Part (2) follows from the Barr-Beck-Lurie theorem
	together with part (1) and the fact that $\forget_\delta$ is conservative and preserves arbitrary colimits, see \cref{Rmk.: properties of fgt to CRan}.
\end{proof}

\begin{theorem}[Animation, {\cite[A.22]{BL}}]\label{Thm.: animated delta rings are the animation}
	The category $\deltaCRan$ of animated $\delta$-rings is equivalent to the animation $\Ani(\deltaCR)$ of the $1$-category of $\delta$-rings.
\end{theorem}

\begin{proof}
	We recall the argument from \cite{BL}: Denote by $F \colon \mathcal{P}_\Sigma(\deltaPoly) \to \deltaCRan$ the left Kan extension of the fully faithful functor $f \colon \deltaPoly \to \deltaCRan$
	of \cref{Cor.: fully faithful embedding of deltaCR}. We claim that $F$ is an equivalence. By \cite[5.5.8.22]{HTT}, it will suffice to show that the essential image of $f$ consists of compact projective objects in $\deltaCRan$ which 
	generate the $\infty$-category of animated $\delta$-rings under sifted colimits.\\
	As $f$ preserves finite coproducts, \cref{Prop.: free animated delta ring} implies that its essential image is equivalent to the essential image of the free animated $\delta$-ring functor $\free_\delta$ restricted to
	$\Poly$. By construction, finitely generated polynomial algebras over $\Z$ are compact projective generators of $\CRan$ and are sent to compact projective objects in $\deltaCRan$ by $\free_\delta$. This uses that its right
	adjoint $\forget_\delta$ preserves sifted colimits.\\
	For the second assertion, consider the smallest full subcategory $\C$ of $\deltaCRan$ which contains the essential image of $f$ and is closed under sifted colimits. $\C$ contains the essential image of $\free_\delta$
	as any animated ring is a sifted colimit in $\CRan$ of elements in $\Poly$. \cite[4.7.3.14]{HA} implies that animated $\delta$-rings are generated by the essential image of
	$\free_\delta$ under geometric realizations which shows that $\C = \deltaCRan$.
\end{proof}

\begin{corollary}
	For any animated $\delta$-ring $(B,\phi_B,h_B)$ the $\delta$-structure on $\pi_0(B) \in \deltaCR$ is given by the map $\pi_0(\delta_B)$ of \cref{Constr.: delta-map on underlying anima}.
\end{corollary}

\begin{proof}
	The assignment $(B,\phi_B,h_B) \mapsto ([B],\delta_B)$ defines a functor $\theta \colon \deltaCRan \to (\An)^{B\Nadd}$ which satisfies ${\ev} \circ \theta \simeq \Map_{\deltaCRan}(\delta\{x\},-)$. Denote the analogous functor
	for ordinary $\delta$-rings by $\theta'$. We show that the diagram
	\[\begin{diag}
		\deltaCRan \ar[d, "\theta"'] \ar[r, "\pi_0"] & \deltaCR \ar[d, "{\theta'}"] \\[4ex]
		(\An)^{B\Nadd} \ar[r, "\pi_0"] & (\Set)^{B\Nadd}
	\end{diag}\]
	in $\Catinf$ commutes: This follows immediately from \cref{Thm.: animated delta rings are the animation} since all functors preserve sifted colimits and we have an equivalence ${\pi_0} \circ \theta \simeq {\theta'} \circ \pi_0$ on $\deltaPoly$.
\end{proof}

\newpage
\section{Animated $\lambda$-rings}\label{Sect.: Animated lambda rings}

The following sections deal with animated $\lambda$-rings and lead to the main theorem (\cref{Thm.: lambdaCRan is the animation}) of the paper. We start with formally defining the $\infty$-category
of animated $\lambda$-rings and showing that ordinary $\lambda$-rings embed as a full subcategory.

\begin{notation}
	We refer to the multiplicative monoid of natural numbers by $\Nmult$ and denote the submonoid generated by the primes $q$ unequal to
	fixed prime $p$ by $\langle p^c \rangle$. This contains all natural numbers $n > 0$ which are coprime to $p$.
\end{notation}

\begin{definition}[Animated $\lambda$-rings]\label{Def.: animated lambda-rings}
	We define the category $\lambdaCRan$ of animated $\lambda$-rings as the pullback
	\[\begin{diag}[column sep=4em]
		\lambdaCRan \ar[r] \ar[d] \ar[dr, phantom, "\lrcorner", near start]& \prod_p (\CRFpan)^{B\langle p^c \rangle} \ar[d, "\prod_p \Fr"] \\[5ex]
		(\CRan)^{B\Nmult} \ar[r, "\prod_p (-) \otimes^L \Fp"] & \prod_p (\CRFpan)^{B\Nmult}
	\end{diag}\]
	in $\Catinf$. Here, the products on the right are indexed over the set of all positive primes $p$.
\end{definition}

\begin{remark}
	Unwinding \cref{Def.: animated lambda-rings} gives the notion of animated $\lambda$-ring introduced in \cref{Sect.: Introduction}: Objects in $\lambdaCRan$ correspond to triples $(A, \phi_A, (h_{(A,p)})_p)$ consisting of an animated ring
	$A$ together with an action $\phi_A\colon B\Nmult \to \CRan$ of the multiplicative monoid of natural numbers on $A$ and a family of homotopies
	\[\begin{diag}[column sep=2.7em]
		&  |[alias=Z]| A \ar[dr, "\red_A"] & \\[3ex]
		A \ar[ur, "\phi_A(p)"] \ar[rr, "\Fr \circ \red_A"', ""{name=U, above}] \ar[r, from=Z, to=U, phantom, "h_{(A,p)}" {font=\footnotesize, yshift=-0.9ex, xshift=0.7ex}]&& A \otimes^L \Fp
	\end{diag}\]
	in $\Map_{(\CRan)^{B\langle p^c \rangle}}(A,A\otimes^L \Fp)$ which witness $\phi_A(p) \in \Fun(B\Nadd,\Fun(B \langle p^c \rangle ,\CRan))$ as a lift of the $p$-Frobenius for each prime $p$.
	Here, the above diagram lives in the category $(\CRan)^{B \langle p^c \rangle}$ of animated rings equipped with an action of $\langle p^c \rangle$, that is both $A$ and $A \otimes^L \Fp$ carry the actions $\phi_A$ and
	$\phi_A \otimes^L \id_{\Fp}$ restricted to the submonoid of $\Nmult$ generated by the complement of $p$.
\end{remark}

\begin{warning}
	It is essential to retain the restricted actions in the definition of animated $\lambda$-rings. This ensures that each endomorphism $\phi_A(p)$ (viewed as an endomorphism of $A$ in $\CRan$) defines a $p$-typical
	$\delta$-structure on $A$ which is respected by the $\phi_A(q)$ in the sense that each is a morphism of animated $\delta$-rings relative to $p$.
	Discarding the restricted actions in \cref{Def.: animated lambda-rings} would lead to the weaker notion of an animated ring equipped with a family of $p$-typical $\delta$-structures for all primes $p$ whose associated Frobenius
	lifts pairwise commute.
\end{warning}

\begin{remark}\label{Rmk.: properties lambdaCRan}
	The diagram in \cref{Def.: animated lambda-rings} is a pullback square in $\Prl$. Thus, the category $\lambdaCRan$ of animated $\lambda$-rings is presentable and
	comes with a forgetful functor
	\begin{equation*}
		\forget_\lambda \colon \lambdaCRan \to (\CRan)^{B\Nmult} \xrightarrow{\ev} \CRan
	\end{equation*}
	to animated rings which is conservative and preserves small colimits. Inserting \cref{Def.: deltapCRan via Frobenius lifts} of $p$-typical animated $\delta$-rings gives a second pullback square 
	\begin{equation}\label{Eq.: animated lambda rings via delta rings}
	\begin{diag}[column sep=2em]
		\lambdaCRan \ar[r] \ar[d] \ar[dr, phantom, "\lrcorner", near start]& \prod_p (\deltapCRan)^{B\langle p^c \rangle} \ar[d, "\prod_p \forget_\delta"] \\[4ex]
		(\CRan)^{B\Nmult} \ar[r, "\Delta"] & \prod_p (\CRan)^{B\Nmult}
	\end{diag}
	\end{equation}
	in $\Catinf$ which belongs to $\Prr$ by \cref{Cor.: Free delta functor}.	
	This shows that the forgetful functor $\forget_\lambda$ is both a left and right adjoint. It follows formally from the Barr-Beck-Lurie theorem that the forgetful functor exhibits $\lambdaCRan$ as both monadic and
	comonadic over animated rings.
\end{remark}

\begin{remark}[Mapping spaces in $\lambdaCRan$]
	A map between animated $\lambda$-rings $(A, \phi_A, h_{(A,p)})$ and $(B, \phi_B,h_{(B,p)})$ is given by the datum $(f,(\tau_p)_p)$ consisting of
	\begin{enumerate}
	\item A natural transformation $f \colon (A,\phi_A) \to (B,\phi_B)$ in $(\CRan)^{B\Nmult}$. This identifies with a functor $\Delta^1 \times B\Nadd \to (\CRan)^{B\langle p^c \rangle}$
		showing that $f$ corresponds to a commutative diagram
				\[\begin{diag}[column sep=4em]
			A \ar[r, "\phi_A(p)"] \ar[d, "f"'] \ar[dr, phantom, "\sigma_p" {font=\footnotesize}]& A \ar[d, "f"] \\[4ex]
			B  \ar[r, "\phi_B(p)"] & B
		\end{diag}\]
		in $(\CRan)^{B\langle p^c \rangle}$ for any prime $p$.
	\item For each prime $p$ a $2$-homotopy $\tau_p$ which witnesses commutativity of the diagram
		\[\begin{diag}[column sep=-1em]
			{{\red_B} \circ f \circ \phi_A(p)} \ar[rr, "\sim"', "{{\red_B} \circ \sigma_p}"] \ar[dr, "\sim", "{\text{induced by $h_{(A,p)}$}}"'] && {{\red_B} \circ {\phi_B(p)} \circ f} \ar[dl, "\sim"', "{h_{(B,p)} \circ f}"] \\[3ex]
			& {{\Fr} \circ {\red_B} \circ f}
		\end{diag}\]
		 in $\Map_{(\CRan)^{B\langle p^c \rangle}}((A,\restr{\phi_A}{\langle p^c \rangle}),(B \otimes^L \Fp,\restr{(\phi_B \otimes^L \id_{\Fp})}{\langle p^c \rangle}))$.
	\end{enumerate}
	This follows immediately from \cref{Eq.: animated lambda rings via delta rings} together with the description of maps between animated $\delta$-rings, see \cref{Prop.: mapping spaces in deltapCRan}.
\end{remark}

\begin{notation}
	We will denote the coproduct in $\lambdaCRan$ again by $\otimes^L$.
\end{notation}

\begin{proposition}\label{Prop.: fully faithful embedding of lambdaCR}
	The inclusion $\iota \colon \CR \to \CRan$ induces a fully faithful functor $\lambdaCR \to \lambdaCRan$ which preserves finite coproducts when restricted to $\lambdaPoly$.
\end{proposition}

\begin{proof}
	Replacing animated rings by ordinary commutative rings in \cref{Eq.: animated lambda rings via delta rings} gives a definition of the $1$-category $\lambdaCR$ of $\lambda$-rings, see
	\cite[4.10]{Rezk}. It thus suffices to show that the inclusion induces a natural transformation of diagrams $\Lambda^2_2 \to \Catinf$ whose components are fully faithful. This follows since
	the inclusion induces a fully faithful functor $\deltaCR \to \deltaCRan$ by \cref{Cor.: fully faithful embedding of deltaCR} which commutes with $\forget_\delta$.\\
	The assertion concerning finite coproducts follows analogously to the proof of \cref{Cor.: fully faithful embedding of deltaCR}: It suffices to show the statement for the composite
	${\forget_\lambda} \circ \iota \colon \lambdaPoly \to \lambdaCRan \to \CRan$ which is equivalent to $\iota \circ \forget \colon \lambdaPoly \to \CR \to \CRan$ by construction. But the underlying
	commutative ring of $A \in \lambdaPoly$ is a polynomial ring on a countably infinite set of variables, see \cite[\S 2]{AtiyahTall} or \cite[Théorème 3]{JoyalLambda} and the inclusion $\iota$ preserves
	small coproducts of $\Z[X]$.
\end{proof}

\begin{proposition}\label{Prop.: pi0 is left adjoint to inclusion of lambdaCR}
	The fully faithful functor $\lambdaCR \to \lambdaCRan$ of \cref{Prop.: fully faithful embedding of lambdaCR} admits a left adjoint which is given by the assignment $A \mapsto \pi_0(A)$ on underlying
	animated rings.
\end{proposition}

\begin{proof}
	By \cref{Prop.: pi0 is left adjoint to inclusion of deltaCR}, the natural transformation defining the inclusion of \cref{Prop.: fully faithful embedding of lambdaCR} admits pointwise left adjoints which are induced by $\pi_0$ and satisfy
	the mate condition. The claim thus follows from the description of adjoints in functor $(\infty,2)$-categories \cite[Corollary 5.2.9]{Rune24} together with the fact that
	$\lim$ preserves adjunctions as a functor of $(\infty,2)$-categories.
\end{proof}

\newpage
\section{A Fracture Square}\label{Sect.: Fracture square}

In this section, we lift a standard fracture square from $\DZ$ to animated $\lambda$-rings (\cref{Lem.: fracture square in lambdaCRan}). We will make use of the notion of idempotent objects and algebras in monoidal $\infty$-categories as defined in \cite[4.8.2.1 and 4.8.2.8]{HA} throughout the section. We begin with a study of animated $\lambda$-$\Z_{(p)}$-algebras and $\lambda$-$\Q$-algebras:

\begin{definition}\label{Def.: animated algebras}
	Let $A$ be an animated $\lambda$-ring. We define the category $\lambdaCRAan$ of animated $\lambda$-$A$-algebras as the slice $(\lambdaCRan)_{A/}$ under $A$.
\end{definition}

\begin{notation}\label{Rmk.: coCartesian symm mon structure}
	We equip the categories $\CRan$, $\deltaCRan$ and $\lambdaCRan$ with the cocartesian symmetric monoidal structure for the remainder of this section but suppress this from notation. In particular, we have
	equivalences $\CAlg(\CRan) \simeq \CRan$ and $\Mod_A(\CRan) \simeq (\CRan)_{A/}$ for any animated ring $A$ by \cite[2.4.3.10]{HA} and \cite[3.4.1.7]{HA}. Analogous results hold for animated $\delta$-rings and $\lambda$-rings.
\end{notation}

\begin{lemma}\label{Rmk.: Zp and Q idempotent in DZ}
	The commutative rings $\Z_{(p)}$ and $\Q$ are idempotent algebras in $\DZ_{\geq 0}$.
\end{lemma}

\begin{proof}
	This follows immediately from flatness of $\Z_{(p)}$ and $\Q$ over $\Z$: The map $\Z_{(p)} \to \Z_{(p)} \otimes^L \Z_{(p)}$ in $\DZ_{\geq 0}$ which is induced by the unit $e \colon \Z \to \Z_{(p)}$ is equivalent to the map
	$\Z_{(p)} \to \Z_{(p)} \otimes_\Z \Z_{(p)}$ of commutative rings sending $x \in \Z_{(p)}$ to $1 \otimes x$. This is an equivalence with inverse is given by multiplication. The same holds true for $\Q$.
\end{proof}

\begin{notation}
	In the following, we view the commutative rings $\Z$, $\Z_{(p)}$ and $\Q$ as animated $\lambda$-rings via \cref{Prop.: fully faithful embedding of lambdaCR}. Note that the (animated) $\lambda$-structures on $\Z$, $\Z_{(p)}$ and $\Q$ are
	essentially unique and determined by each Frobenius lift being an identity. We will write $\lambdaCRpan$ for the category of animated $\lambda$-$\Z_{(p)}$-algebras.
\end{notation}

\begin{notation}
	We write $\deltaCRpan$ for the category of animated $\delta$-$\Z_{(p)}$-algebras defined as the slice under $\Z_{(p)}$ viewed as an animated $p$-typical $\delta$-ring via the forgetful functor
	$\lambdaCRan \to \deltapCRan$.
\end{notation}

\begin{proposition}[Idempotent algebras]\label{Prop.: idempotent algebras}
	The animated $\lambda$-rings $\Z_{(p)}$ and $\Q$ are idempotent algebras in $\lambdaCRan$, $\deltaCRan$ and $\CRan$.
\end{proposition}

\begin{proof}
	This follows from \cref{Rmk.: Zp and Q idempotent in DZ} since the forgetful functors to $\DZ_{\geq 0}$ are symmetric monoidal and conservative.
\end{proof}

\begin{proposition}\label{Prop.: localisations lambdap and lambdaQ}
	Denote by $e \colon \Z \to \Z_{(p)}$ and $e' \colon \Z \to \Q$ the unique maps of animated $\lambda$-rings.
	\begin{enumerate}
		\item The category of animated $\lambda$-$\Z_{(p)}$-algebras is equivalent to the full subcategory of $\lambdaCRan$ spanned by those animated $\lambda$-rings $A$ for which the map
		$\eta_{A,(p)} \colon A \to A \otimes^L \Z_{(p)}$ induced by ${\id_A} \otimes^L e$ is an equivalence in $\lambdaCRan$.
		\item The category $\lambdaCRpan$ is a localisation of the category of animated $\lambda$-rings. We write
				\begin{equation*}
				\begin{diag}
				\lambdaCRan \ar[r, bend left=20, ""{name=U, below}, "L_{(p)}"] & \lambdaCRpan \ar[l, hook', bend left=20, ""{name=D, above}, "\iota_{(p)}" {yshift=-0.4ex, xshift=2.5ex}]
				\ar[from=U, to=D, phantom, "\dashv" {rotate=270, yshift=0.3ex, xshift=-0.1ex}]
				\end{diag}
				\end{equation*}
				for the adjunction induced by the idempotent monad $\iota_{(p)} \circ L_{(p)} \simeq (-) \otimes^L \Z_{(p)}$ on $\lambdaCRan$ with unit $\eta_A$.
		\item The category $\lambdaCRQan$ of animated $\lambda$-$\Q$-algebras is equivalent to the full subcategory of $\lambdaCRan$ spanned by those $A \in \lambdaCRan$ for which the map
			$\eta_{A,\Q} \colon A \to A \otimes^L \Q$ induced by $e'$ is an equivalence of animated $\lambda$-rings. We write
			\[\begin{diag}
				\lambdaCRan \ar[r, bend left=20, ""{name=U, below}, "L_{\Q}"] & \lambdaCRQan \ar[l, hook', bend left=20, ""{name=D, above}, "\iota_\Q" {yshift=-0.4ex, xshift=0.5ex}]
				\ar[from=U, to=D, phantom, "\dashv" {rotate=270, yshift=0.3ex, xshift=-0.1ex}]
			\end{diag}\]
			for the localisation with unit $\eta_\Q$ which is induced by the idempotent monad $\iota_\Q \circ L_{\Q} \simeq (-) \otimes^L \Q$ on $\lambdaCRan$.
	\end{enumerate}
\end{proposition}

\begin{proof}
	This follows from \cref{Prop.: idempotent algebras} and \cite[4.8.2.10]{HA}.
\end{proof}

\begin{remark}
	The statements of \cref{Prop.: localisations lambdap and lambdaQ} carry over verbatim to the $\infty$-categories $\CRpan$, $\CRQan$ and $\deltaCRpan$.
\end{remark}

\begin{remark}\label{Rmk.: animated lambda-Q-algebras are lambda-Zp-algebras}
	We have an equivalence $\Q \otimes^L \Z_{(p)} \simeq \Q$ in $\lambdaCRan$ for any prime $p$. It follows that animated $\lambda$-$\Q$-algebras form a full subcategory of the category of animated $\lambda$-$\Z_{(p)}$-algebras
	for all primes $p$. We denote the fully faithful inclusion by
	\begin{equation*}
		\iota_\Q^{(p)} \colon \lambdaCRQan \to \lambdaCRpan.
	\end{equation*}
	Note that $\iota_\Q^{(p)} $ admits a left adjoint which is given by the composite ${L_\Q} \circ \iota_{(p)}$, see \cref{Prop.: localisations lambdap and lambdaQ} for notation.
\end{remark}

\begin{notation}
	Let $\C$ and $\D$ be $\infty$-categories and assume that $\D$ admits a terminal object $0$. We say that a functor $F \colon \C \to \D$ is \textit{null-homotopic} if it factors through $0 \colon \Delta^0 \to \D$ in $\Catinf$.
\end{notation}

\begin{lemma}\label{Lem.: p-local tensored with Fq is zero}
	The functor ${(-) \otimes^L \Fq} \circ \iota_{(p)} \colon \CRpan \hookrightarrow \CRan \to \CRFqan$ is null-homotopic for any pair of distinct primes $p$ and $q$.
\end{lemma}

\begin{proof}
	The functor $\theta = {(-) \otimes^L \Fq} \circ \iota_{(p)}$ preserves sifted colimits and is thus left Kan extended from its restriction to finitely generated polynomial $\Z_{(p)}$-algebras. But $(-) \otimes^L \Fq$ is constant on $\Poly_{\Z_{(p)}}$:
	The standard resolution of $\Fq$ shows
	\begin{equation*}
		\Z_{(p)}[x_1,\ldots,x_n] \otimes^L_\Z \Fq \simeq 0
	\end{equation*}
	underlying in $\DZ_{\geq 0}$ for any $n \geq 0$ and the forgetful functor from animated $\Fq$-algebras to connective chain complexes reflects terminal objects. %all limits as it is conservative and preserves them
	It thus follows that the functor
	\begin{equation*}
		\Poly_{\Z_{(p)}}/A \xrightarrow{\forget} \Poly_{\Z_{(p)}} \xrightarrow{\theta \circ j} \CRFqan
	\end{equation*}
	is null-homotopic for any $A \in \CRpan$. Further note that the functor $\Poly_{\Z_{(p)}}/A \to \Delta^0$ is right cofinal since the slice admits an initial object. The claim now follows from the colimit formula for pointwise left Kan extensions.
\end{proof}

\begin{corollary}\label{Cor.: tensoring with Fq is zero on Q-algebras}
	The composite ${(-) \otimes^L \Fq} \circ \iota_\Q \colon \CRQan \hookrightarrow \CRan \to \CRFqan$ is null-homotopic for any prime $q$.
\end{corollary}

\begin{proof}
	This follows from \cref{Lem.: p-local tensored with Fq is zero} together with the observation that the inclusion $\iota_\Q$ factors through $\CRpan$ for any prime $p$, see \cref{Rmk.: animated lambda-Q-algebras are lambda-Zp-algebras}.
\end{proof}

We next prove the classically known descriptions of $\lambda$-$\Z_{(p)}$-algebras in terms of $\delta$-$\Z_{(p)}$-algebras and $\lambda$-$\Q$-algebras in terms of $\Q$-algebras in the animated setting:

\begin{lemma}\label{Lem.: pullback idempotent algebras}
	Let $\C^\otimes$ and $\D^\otimes$ be symmetric monoidal $\infty$-categories and let $F \colon \C^\otimes \to \D^\otimes$ be a symmetric monoidal and conservative functor. Let $A \in \CAlg(\C)$ be an
	idempotent algebra in $\C$. Then the commutative diagram
	\[\begin{diag}
		\Mod_A(\C) \ar[r] \ar[d, "F'"] & \C \ar[d, "F"] \\[3ex]
		\Mod_{F(A)}(\D) \ar[r] & \D
	\end{diag}\]
	is a pullback square in $\Catinf$. Here, the horizontal functors forget the module structures and $F'$ is induced by $F$.
\end{lemma}

\begin{proof}
	It is clear that $F(A) \in \CAlg(\D)$ is an idempotent algebra. By \cite[4.8.2.10]{HA}, the forgetful functors are fully faithful and induce equivalences $\Mod_A(\C) \simeq L_A\C$ and $\Mod_{F(A)}(\D) \simeq L_{F(A)}\D$ of symmetric
	monoidal $\infty$-categories. Here, $L_A\C$ denotes the essential image of the localisation functor determined by $A$ which coincides with the full subcategory of $\C$ spanned by those $C \in \C$ for which the unit
	$\eta_{A,C} \colon C \to L_A(C) \simeq C \otimes A$ is an equivalence in $\C$. It thus suffices to see that the induced functor to the pullback is essentially surjective. For this, observe that objects in the pullback can be identified with
	triples $(D,C,\alpha)$ where $D \in L_{FA}(\D)$, $C \in \C$ and $\alpha \colon D \simeq F(C)$ is an equivalence in $\D$. But this immediately shows that $C \in L_A(\C)$ since $F(\eta_{A,C}) \simeq \eta_{FA,FC} \simeq \eta_{FA,D}$
	is an equivalence in $\D$.
\end{proof}

\begin{lemma}[Animated $\lambda$-$\Z_{(p)}$-algebras]\label{Lem.: description animated lambda-Zp-algebras}
	 There exists an equivalence $\lambdaCRpan \simeq (\deltaCRpan)^{B\langle p^c \rangle}$ of $\infty$-categories for any prime $p$.
\end{lemma}

\begin{proof}
	Fix a prime $p$. The forgetful functor $\lambdaCRan \to (\CRan)^{B\Nmult}$ preserves small coproducts and is conservative. Thus, it induces a pullback square
	\begin{equation}\label{Eq.: prelim pb for animated lambda Zp}
	\begin{diag}[column sep=2em]
		\lambdaCRpan \ar[r, "\iota_{(p)}"] \ar[d]  \ar[dr, phantom, "\lrcorner", near start] & \lambdaCRan \ar[d] \\[4ex]
		(\CRpan)^{B\Nmult} \ar[r, "\iota_{(p)}"] & (\CRan)^{B\Nmult}
	\end{diag}
	\end{equation}
	in $\Catinf$ by \cref{Lem.: pullback idempotent algebras} applied to $\Z_{(p)}$. Here, we identify $\Mod_{\Z_{(p)}}((\CRan)^{B\Nmult}) \simeq (\CRpan)^{B\Nmult}$.
	Pasting \cref{Eq.: prelim pb for animated lambda Zp} and the pullback square from \cref{Def.: animated lambda-rings} gives a second pullback square
	\begin{equation}\label{Diag.: pullback describing p-local global delta rings}
	\begin{diag}[column sep=2em]
		\lambdaCRpan \ar[r] \ar[d] \ar[dr, phantom, "\lrcorner", near start] & \prod_q (\CRFqan)^{B\langle q^c \rangle} \ar[d, "\prod_q \Frq"] \\[5ex]
		(\CRpan)^{B\Nmult} \ar[r, "\alpha"]  & \prod_q (\CRFqan)^{B\Nmult}.
	\end{diag}
	\end{equation}
	in $\Catinf$. Here, $\alpha$ is induced by the functors ${(-) \otimes^L \Fq} \circ \iota_{(p)} \colon \CRpan \to \CRFqan$ ranging over all primes $q$. Recall that ${(-) \otimes^L \Fq} \circ \iota_{(p)}$ is null-homotopic
	for any $q$ distinct from $p$ by \cref{Lem.: p-local tensored with Fq is zero}. It immediately follows that $\alpha$ is equivalent to the composite
	 \begin{equation*}
	 	(\CRpan)^{B\Nmult} \xrightarrow{{(-) \otimes^L \Fp} \circ \iota_{(p)}} (\CRFpan)^{B\Nmult} \xrightarrow{(\id,0)} \textstyle \prod_q (\CRFqan)^{B\Nmult}
	 \end{equation*}
	 where $(\id,0)$ denotes the functor induced by the identity for $q = p$ and a null-homotopic functor for any prime $q$ unequal to $p$. As the Frobenius reflects terminal objects the diagram
	 \[\begin{diag}
		(\CRFpan)^{B\langle p^c \rangle} \ar[d, "\Fr"] \ar[r, "{(\id,0)}"] &[-1ex] \prod_q (\CRFqan)^{B\langle q^c \rangle} \ar[d, "\prod_q \Frq"]\\[4ex]
		(\CRFpan)^{B\Nmult} \ar[r, "{(\id,0)}"] & \prod_q (\CRFqan)^{B\Nmult}
	 \end{diag}\]
	is a pullback square. Inserting the definition of animated $\delta$-rings thus gives a pullback square
	\[\begin{diag}
		\lambdaCRpan \ar[r] \ar[d] \ar[dr, phantom, "\lrcorner", near start]& (\deltaCRan)^{B\langle p^c \rangle} \ar[d] \\[4ex]
		(\CRpan)^{B\Nmult} \ar[r,"\iota_{(p)}"] & (\CRan)^{B\Nmult}
	\end{diag}\]
	in $\Catinf$ whose limit is equivalent to $(\deltaCRpan)^{B\langle p^c \rangle}$ by \cref{Lem.: pullback idempotent algebras} applied to the forgetful functor $\deltaCRan \to (\CRan)^{B\Nadd}$.
\end{proof}

\begin{corollary}[Animated $\lambda$-$\Q$-algebras]\label{Cor.: animated lambda-Q-algebras}
	There exists an equivalence $\lambdaCRQan \simeq (\CRQan)^{B\Nmult}$ of $\infty$-categories.
\end{corollary}

\begin{proof}
	Fix a prime $p$. \Cref{Lem.: pullback idempotent algebras} applied to the idempotent algebras $\Z_{(p)}$ and $\Q$ in $\lambdaCRan$ shows that there exists a pullback square
	\[\begin{diag}[column sep=2em]
		\lambdaCRQan \ar[d] \ar[r, "\iota_\Q^{(p)}"] & \lambdaCRpan \ar[d] \\[4ex]
		(\CRQan)^{B\Nmult} \ar[r, "\iota_\Q^{(p)}"] & (\CRpan)^{B\Nmult}
	\end{diag}\]
	of $\infty$-categories. Using the equivalence $\lambdaCRpan \simeq (\deltaCRpan)^{B\langle p^c \rangle}$ of \cref{Lem.: description animated lambda-Zp-algebras} and inserting the definition of $p$-typical animated $\delta$-rings
	gives a second pullback square
	\[\begin{diag}
		\lambdaCRQan \ar[d] \ar[r] & (\CRFpan)^{B\langle p^c \rangle} \ar[d, "\Fr"] \\[4ex]
		(\CRQan)^{B\Nmult} \ar[r] & (\CRFpan)^{B\Nmult}
	\end{diag}\]
	in $\Catinf$ where the lower horizontal functor is induced by ${(-) \otimes^L \Fp} \circ \iota_\Q \colon \CRQan \to \CRFpan$. Since the functor ${(-) \otimes^L \Fp} \circ \iota_\Q$ is null-homotopic by
	\cref{Cor.: tensoring with Fq is zero on Q-algebras} and the Frobenius reflects terminal objects it follows that the left hand map is an equivalence.
\end{proof}

\begin{lemma}[Fracture square]\label{Lem.: fracture square in lambdaCRan}
	For any animated $\lambda$-ring $A$, there exists a pullback square
	\begin{equation}\label{Eq.: fracture square}
	\begin{diag}[column sep=3ex]
	A \ar[r, "\prod_p \eta_{A,(p)}"] \ar[d, "\eta_{A,\Q}"'] \ar[dr, phantom, "\lrcorner" {xshift=-3ex, yshift=1ex}] &[2ex] \prod_p (A \otimes^L \Z_{(p)}) \ar[d, "\eta_{\prod_p (A \otimes^L \Z_{(p)}),\Q}"] \\[4ex]
	A \otimes^L \Q \ar[r] & \left(\prod_p A \otimes^L \Z_{(p)} \right) \otimes^L \Q
	\end{diag}
	\end{equation}
	in $\lambdaCRan$. Here, $\eta_{(p)}$ and $\eta_\Q$ denote the units of \cref{Prop.: localisations lambdap and lambdaQ}. The lower horizontal map is given by $(\prod_p \eta_{A,(p)}) \otimes^L \id_{\Q}$.
\end{lemma}

\begin{proof}
	The diagram is functorial in $A$ and commutes by naturality of the unit $\eta_\Q$. It suffices to show that \cref{Eq.: fracture square} is a pullback square underlying in $\DZ$. Abbreviate $\prod_p \eta_{A,(p)}$ by $\alpha$. We
	claim that the induced map between fibres $\gamma \colon \fib(\alpha) \to \fib(\alpha \otimes^L \id_{\Q})$ is an equivalence in $\DZ$. This is equivalent to $\fib(\alpha) \in \DQ$ since exactness of
	$(-) \otimes^L \Q$ shows that $\gamma$ is induced by the unit $e \colon \Z \to \Q$ of the idempotent algebra $\Q \in \DZ$. It thus suffices to show that
	$\fib(\alpha) \otimes^L \Fp \simeq 0$ for all primes $p$. But this holds true since
	\[\begin{diag}[column sep=-1em]
	 	\fib(\alpha) \otimes^L \Fp \ar[r] \ar[d] & A \otimes^L \Fp \ar[d, "\alpha \otimes^L \id_{\Fp}"] \\[3.5ex]
		0 \ar[r] & (\prod_q A \otimes^L \Z_{(q)}) \otimes^L \Fp
	\end{diag}\]
	is a pullback square in $\DZ$ for any $p$ and the map $\alpha \otimes^L \id_{\Fp}$ is an equivalence. This uses that small products in $\DZ_{\geq 0}$ coincide with products in $\DZ$ and that
	$\Fp$ is perfect in $\DZ$.
\end{proof}

\newpage
\section{Free animated $\lambda$-rings}\label{Sect.: Free animated lambda rings}

In this section, we prove that our notion of animated $\lambda$-rings (\cref{Def.: animated lambda-rings}) coincides with the animation $\Ani(\lambdaCR)$ of the $1$-category of $\lambda$-rings
(\cref{Thm.: lambdaCRan is the animation}). We reduce this to identifying the free animated $\lambda$-$\Z_{(p)}$-algebra on a single generator with the $p$-localisation of the ordinary free $\lambda$-ring via the
results of \cref{Sect.: Fracture square}.

\begin{notation}\label{Def.: forgetlambdap}
	We denote by $\forgetlambdap$ the forgetful functor from animated $\lambda$-$\Z_{(p)}$-algebras to animated rings. Formally, it is given by the composite
	\begin{equation*}
		{\forget_\lambda} \circ \iota_{(p)} \colon \lambdaCRpan \to \lambdaCRan \to \CRan.
	\end{equation*}
	Analogously, we write $\forgetlambdaQ$ for the functor which assigns an animated $\lambda$-$\Q$-algebra
	its underlying animated ring. Formally, we have $\forgetlambdaQ \simeq {\forget_\lambda} \circ \iota_\Q \colon \lambdaCRQan \to \CRan$.
\end{notation}

\begin{lemma}[Free animated $\lambda$-$\Z_{(p)}$-algebras]\label{Lem.: left adjoint to forgetlambdap}
	The forgetful functor $\forgetlambdap$ admits a left adjoint $\freelambdap$ which satisfies
	\begin{equation*}
		\freelambdap(\Z[X]) \simeq \coprod_{\langle p^c \rangle} (\delta\{x\} \otimes^L \Z_{(p)})
	\end{equation*}
	in $(\deltaCRan)^{B\langle p^c \rangle}$. Here, $\delta\{x\}$ denotes the free animated $p$-typcial $\delta$-ring and $\langle p^c \rangle$ acts by shifting the factors of the coproduct. In particular, the free
	animated $\lambda$-$\Z_{(p)}$-algebra is discrete.
\end{lemma}

\begin{proof}
	\Cref{Lem.: description animated lambda-Zp-algebras} allows to identify the left adjoint $\freelambdap$ with the composite
	\begin{equation*}
		\CRan \xrightarrow{\free_\delta} \deltaCRan \xrightarrow{L_{(p)}} \deltaCRpan \xrightarrow{\Lan} (\deltaCRpan)^{B\langle p^c \rangle} \simeq \lambdaCRpan
	\end{equation*}
	of left adjoints. Recall that $\free_\delta(\Z[X]) \simeq \delta\{x\}$ by \cref{Prop.: free animated delta ring}. Thus, the colimit formula for pointwise left Kan extensions shows
	\begin{equation*}
		\ev \bigbrace{\free_\lambda^{(p)}(\Z[X])} \simeq \coprod_{\langle p^c \rangle} L_{(p)}(\delta\{x\}) \simeq L_{(p)} \bigbrace{\coprod_{\langle p^c \rangle} \delta\{x\} \otimes^L \Z_{(p)} }
	\end{equation*}
	in $\deltaCRpan$. Here, the coproduct on the right is taken in animated $\delta$-rings and $\langle p^c \rangle$ acts by shifting the factors. The equivalence ${\iota_{(p)}} \circ L_{(p)} \simeq (-) \otimes^L \Z_{(p)}$ shows
	\begin{align*}
	\ev \bigbrace{\free_\lambda^{(p)}(\Z[X])} &\simeq \bigbrace{\coprod_{\langle p^c \rangle} \delta\{x\} \otimes^L \Z_{(p)} } \otimes^L \Z_{(p)} \\
	&\simeq \coprod_{\langle p^c \rangle} \bigbrace{ \delta\{x\} \otimes^L \Z_{(p)}}
	\end{align*}
	in $\deltaCRan$. Here, we have used that $(-) \otimes^L \Z_{(p)} \colon \deltaCRan \to \deltaCRan$ preserves arbitrary colimits since $\Z_{(p)}$ is an idempotent algebra in animated $\delta$-rings. The second assertion
	follows as ${\delta\{x\}} \otimes^L \Z_{(p)} \simeq \Z_{(p)}[x_n |\,n \geq 0]$ in $\CRan$ and coproducts of polynomial algebras are discrete.
\end{proof}

\begin{lemma}\label{Lem.: forgetlambdap commutes with Lp}
	There exists an equivalence
	\begin{equation*}
		{\forgetlambdap} \circ {L_{(p)}} \simeq (-) \otimes^L \Z_{(p)} \circ \forget_\lambda
	\end{equation*}
	of functors $\lambdaCRan \to \CRan$ for any prime $p$. Here, we view $(-) \otimes^L \Z_{(p)}$ as an endofunctor on animated rings and denote by $L_{(p)}$ the localisation of \cref{Prop.: localisations lambdap and lambdaQ}.
\end{lemma}

\begin{proof}
	We have ${\forgetlambdap} \circ {L_{(p)}} \simeq {\forget_\lambda} \circ (-) \otimes^L \Z_{(p)}$ by \cref{Prop.: localisations lambdap and lambdaQ}. The claim follows as the forgetful functor from animated $\lambda$-rings to animated
	rings commutes with coproducts.
\end{proof}

\begin{lemma}[Free animated $\lambda$-$\Q$-algebras]\label{Lem.: left adjoint to forgetlambdaQ}
	The forgetful functor $\forgetlambdaQ$ admits a left adjoint which is given by the composite
	\begin{equation*}
		\freelambdaQ \colon \CRan \xrightarrow{L_\Q} \CRQan \xrightarrow{\Lan} (\CRQan)^{B\Nmult} \simeq \lambdaCRQan
	\end{equation*}
	in $\Catinf$. In particular, the free animated $\lambda$-$\Q$-algebra is discrete.
\end{lemma}

\begin{proof}
	This follows immediately from \cref{Cor.: animated lambda-Q-algebras}. We compute $\freelambdaQ(\Z[X]) \simeq \Q[x_n |\, n\geq 1]$ by the colimit formula for pointwise left Kan extensions and observe that
	$m \in \Nmult$ acts via the $\Q$-algebra endomorphism sending $x_n$ to $x_{mn}$.
\end{proof}

\begin{remark}\label{Rmk.: freelambdaQ via p}
	The free animated $\lambda$-$\Q$-algebra functor of \cref{Lem.: left adjoint to forgetlambdaQ} admits a second description in terms of free animated $\lambda$-$\Z_{(p)}$-algebras for any prime $p$:
	\begin{equation*}
		\freelambdaQ \simeq L_\Q \circ \iota_{(p)} \circ \freelambdap.
	\end{equation*}
	This follows from \cref{Rmk.: animated lambda-Q-algebras are lambda-Zp-algebras} together with the equivalence $\forgetlambdaQ \simeq {\forgetlambdap} \circ \iota_\Q^{(p)}$.
\end{remark}

\begin{notation}[Free $\lambda$-rings]\label{Ordinary free lambda rings}
	We recall Joyal's description of the free $\lambda$-ring, see \cite[Theorem 1]{JoyalLambda}: Let $C(P) = \{ (p_1,\ldots,p_n) | \, n \geq 0, \text{$p_i > 0$ prime with $p_1 \leq \ldots \leq p_n$} \}$
	denote the set of non-decreasing sequences of positive primes. There exists a unique global $\delta$-structure (aka $\lambda$-structure) on the polynomial ring
	$\Z[X_\sigma | \,\sigma \in C(P)]$ such that $\delta_\sigma(X_0) = X_\sigma$ for any $\sigma \in C(P)$. Here, we write $X_0$ for the variable corresponding to the empty sequence and set
	$\delta_\sigma = \delta_{p_1} \circ \ldots \circ \delta_{p_n}$ for $\sigma = (p_1,\ldots,p_n)$ with $\delta_0 = \id$. Given a fixed prime $p$, the associated $p$-Frobenius lift $\phi^p$ on
	$\Z[X_\sigma | \,\sigma \in C(P)]$ is the unique ring homomorphism with $\phi^p(X_\sigma) = X_\sigma^p + p(\delta_p \circ \delta_\sigma)(X_0)$ for all $\sigma \in C(P)$. Here, the value
	$(\delta_p \circ \delta_\sigma)(X_0) \in \Z[X_\sigma | \,\sigma \in C(P)]$ is uniquely determined by the commutation formulas for the $\delta$-maps as spelled out in \cite{JoyalLambda}.
	Together with this structure, the ring $\Z[X_\sigma | \,\sigma \in C(P)]$ is the free $\lambda$-ring $\lambda\{x\}$ on the generator $x = X_0$.
\end{notation}

\begin{proposition}\label{Prop.: free lambda ring free p-local}
	The free $\lambda$-ring satisfies $L_{(p)}(\lambda\{x\}) \simeq \free_\lambda^{(p)}(\Z[X])$ in $\lambdaCRpan$ for any prime $p$.
\end{proposition}

\begin{proof}
	\Cref{Lem.: left adjoint to forgetlambdap} identifies the free animated $\lambda$-$\Z_{(p)}$-algebra with the discrete animated $\delta$-ring
	\begin{equation*}
		\Z_{(p)}[X_{(n,m)} | \, n \geq 0, m \in \langle p^c \rangle] \in (\deltaCRpan)^{B\langle p^c \rangle}
	\end{equation*}
	whose $p$-typical $\delta$-structure is uniquely determined by the $\Z_{(p)}$-algebra endomorphism $\phi^p$ defined by $X_{(n,m)} \mapsto X_{(n,m)}^p + pX_{(n+1,m)}$. Any prime
	$q$ unequal to $p$ acts via the $\Z_{(p)}$-algebra endomorphism $\phi^q$ sending $X_{(n,m)}$ to $X_{(n,qm)}$. Observe that $\phi^q$ commutes with $\phi^p$ and thus defines a morphism of
	$\delta$-rings. We claim that the map
	\begin{equation*}
		\Z_{(p)}[X_{(n,m)} | \, n \geq 0, m \in \langle p^c \rangle]  \to L_{(p)}(\lambda\{x_0\}), \,X_{(n,m)} \mapsto \phi^m \delta^n_p (x_0)
	\end{equation*}
	in $(\deltaCRpan)^{B\langle p^c \rangle}$ is an equivalence. It suffices to see that the map defines an isomorphism of underlying $\Z_{(p)}$-algebras. Using the description of the free $\lambda$-ring from
	\cref{Ordinary free lambda rings} gives $L_{(p)}(\lambda\{x_0\}) \simeq \Z_{(p)}[X_\sigma | \,\sigma \in C(P)]$ underlying in $\CRpan$. Injectivity holds since the expression $\phi^m \delta^n_p (x_0)$ contains
	the term $p^nm X_{\sigma(n,m)}$ for $\sigma(n,m)$ the unique increasing sequence of primes corresponding to the natural number $p^nm$ and the $X_\sigma$ are algebraically independent over $\Z_{(p)}$.
	Surjectivity can be shown by induction on the length of $\sigma$ using that $\phi^q(X_0) = X_0^q + qX_{(q)}$ with $q$ invertible in $L_{(p)}(\lambda\{x_0\})$.
\end{proof}

\begin{corollary}\label{Cor.: free lambda ring free rationally}
	The free $\lambda$-ring $\lambda\{x\}$ satisfies $L_{\Q}(\lambda\{x\}) \simeq \free_\lambda^{\Q}(\Z[X])$ in $\lambdaCRQan$.
\end{corollary}

\begin{proof}
	We have equivalences
	\begin{align*}
		\free_\lambda^{\Q}(\Z[X]) &\simeq L_{\Q}\iota_{(p)}\free_\lambda^{(p)}(\Z[X]) \\[1.2ex]
		& \simeq L_{\Q}\iota_{(p)} L_{(p)}(\lambda\{x\}) \\[1.2ex]
		& \simeq L_{\Q}(\lambda\{x\} \otimes^L \Z_{(p)})
	\end{align*}
	by \cref{Rmk.: freelambdaQ via p,Prop.: free lambda ring free p-local}. The equivalence $\Q \otimes^L \Z_{(p)} \simeq \Q$
	shows $L_{\Q}(\lambda\{x\} \otimes^L \Z_{(p)}) \simeq L_\Q(\lambda\{x\})$ as animated $\lambda$-$\Q$-algebras.
\end{proof}

\begin{proposition}[Free animated $\lambda$-rings]\label{Prop.: free animated lambda ring}
	The free $\lambda$-ring $\lambda\{x\}$ corepresents the functor
	\begin{equation*}
		\Map_{\CRan}(\Z[X],\forget_\lambda(-)) \colon \lambdaCRan \to \An,
	\end{equation*}
	i.e. it is the free animated $\lambda$-ring on a single generator $x$.
\end{proposition}

\begin{proof}
	Fix an animated $\lambda$-ring $A$. The fracture square of \cref{Lem.: fracture square in lambdaCRan} induces a pullback square
	\[\begin{diag}[column sep=2ex]
		\Map_{\lambdaCRan}(\lambda\{x\},A) \ar[r] \ar[d] \ar[dr, phantom, "\lrcorner", near start] & \Map_{\lambdaCRan}\bigbrace{\lambda\{x\},\prod_p A \otimes^L \Z_{(p)}} \ar[d] \\[4ex]
		\Map_{\lambdaCRan}\bigbrace{\lambda\{x\},A \otimes^L \Q} \ar[r] & \Map_{\lambdaCRan}\bigbrace{\lambda\{x\},(\prod_p A \otimes^L \Z_{(p)}) \otimes^L \Q}
	\end{diag}\]
	of mapping anima. \Cref{Prop.: localisations lambdap and lambdaQ} together with \cref{Lem.: forgetlambdap commutes with Lp,Prop.: free lambda ring free p-local} give the following equivalences
	\begin{align*}
		\Map_{\lambdaCRan}\bigbrace{\lambda\{x\},\prod_p A \otimes^L \Z_{(p)}} &\simeq
		\prod_p \Map_{\lambdaCRan}(\lambda\{x\},A \otimes^L \Z_{(p)}) \\
		&\simeq \prod_p \Map_{\lambdaCRpan}(L_{(p)}\lambda\{x\},L_{(p)}A) \\
		&\simeq \prod_p \Map_{\CRan}(\Z[X],({\forgetlambdap} \circ L_{(p)})(A)) \\
		&\simeq \Map_{\CRan}\bigbrace{\Z[X],\prod_p \forget_\lambda (A) \otimes^L \Z_{(p)}}
	\end{align*}
	which are natural in $A \in \lambdaCRan$. Likewise, \cref{Cor.: free lambda ring free rationally} and the analogue of \cref{Lem.: forgetlambdap commutes with Lp} for $\forgetlambdaQ$ show that there
	exist natural equivalences
	\begin{align*}
		\Map_{\lambdaCRan}(\lambda\{x\},A \otimes^L \Q) &\simeq \Map_{\CRan}(\Z[X],\forget_\lambda( A) \otimes^L \Q) \\[1ex]
		\Map_{\lambdaCRan}\bigbrace{\lambda\{x\},(\prod_p A \otimes^L \Z_{(p)}) \otimes^L \Q} &\simeq \Map_{\CRan}\bigbrace{\Z[X],\bigbrace{\prod_p (\forget_\lambda(A) \otimes^L \Z_{(p)}} \otimes^L \Q}
	\end{align*}
	of mapping anima. The claim now follows immediately from \cref{Lem.: fracture square in lambdaCRan} together with the fact that $\forget_\lambda$ preserves small limits and colimits.
\end{proof}

\begin{theorem}[Animation]\label{Thm.: lambdaCRan is the animation}
	The category $\lambdaCRan$ of animated $\lambda$-rings is equivalent to the animation $\Ani(\lambdaCR)$ of the ordinary category of $\lambda$-rings.
\end{theorem}

\begin{proof}
	The proof is essentially identical to the case of animated $\delta$-rings, see the proof of \cref{Thm.: animated delta rings are the animation} for details: \cite[5.5.8.22]{HTT} implies that the left derived functor
	 $F \colon \mathcal{P}_\Sigma(\lambdaPoly) \to \lambdaCRan$ of the fully faithful inclusion $\lambdaCR \to \lambdaCRan$ of \cref{Prop.: fully faithful embedding of lambdaCR} is an equivalence
	of $\infty$-categories. This relies on the description of free animated $\lambda$-rings given in \cref{Prop.: free animated lambda ring}.
\end{proof}

\begin{notation}
	Let $A$ be an ordinary $\lambda$-ring. We write $\lambdaPolyA$ for the category of free $\lambda$-$A$-algebras on finitely many generators and set $\Ani(\lambdaCR_A) = \mathcal{P}_\Sigma(\lambdaPolyA)$.
\end{notation}

\begin{corollary}[Relative animation]
	There exists an equivalence $\lambdaCRAan \simeq \Ani(\lambdaCR_A)$ of $\infty$-categories for any ordinary $\lambda$-ring $A \in \lambdaCR$.
\end{corollary}

\begin{proof}
	Fix $A \in \lambdaCR$. The equivalence $\lambdaCRan \simeq \Ani(\lambdaCR)$ of \cref{Thm.: lambdaCRan is the animation} induces an equivalence $\lambdaCRAan \simeq \Ani(\lambdaCR)_{A/}$
	of slice $\infty$-categories by \cite[1.2.9.3]{HTT} and \cite[2.4.5]{HTT}. It thus suffices to show that there exists an equivalence $\Ani(\lambdaCR)_{A/} \simeq \Ani(\lambdaCR_A)$.
	This follows from \cite[25.1.4.2]{SAG} applied to the unique morphism $\Z \to A$ of $\lambda$-rings. Here we use that both the statement and proof of \cite[25.1.4.2]{SAG} carry over verbatim to the setting of $\lambda$-rings.
\end{proof}

\begin{corollary}\label{Cor.: discrete animated lambda-rings}
	The category $\lambdaCR$ of ordinary $\lambda$-rings is equivalent to the full subcategory $\tau_{\leq 0}(\lambdaCRan)$ of $0$-truncated objects in $\lambdaCRan$. These are precisely those animated $\lambda$-rings
	whose underlying animated ring is discrete.
\end{corollary}

\begin{proof}
	\Cref{Thm.: lambdaCRan is the animation} and \cite[5.5.8.26]{HTT} show that there exist the following equivalences
	\begin{equation*}
		\lambdaCR \simeq \Fun^{\prod}(\lambdaPoly^{\op},\tau_{\leq 0}\An) \simeq \tau_{\leq 0}\mathcal{P}_\Sigma(\lambdaPoly) \simeq \tau_{\leq 0}(\lambdaCRan)
	\end{equation*}
	of $\infty$-categories. It follows that the functor $\pi_0 \colon \lambdaCRan \to \lambdaCR$ of \cref{Prop.: pi0 is left adjoint to inclusion of lambdaCR} coincides with $0$-truncation.
	In particular, we see that an animated $\lambda$-ring $A$ is $0$-truncated if and only if the unit $A \to \pi_0(A)$ is an equivalence in $\lambdaCRan$. The second assertion now follows from
	observing that the forgetful functor $\forget_\lambda \colon \lambdaCRan \to \CRan$ is conservative and commutes with both $\pi_0$ and the inclusion.
\end{proof}

\newpage
\section{Animated $\lambda$-rings via global Witt vectors}\label{Sect.: Global animated Witt vectors}

The functor $W \colon \CR \to \CR$ of global Witt vectors extends to a comonad on commutative rings, see \cite[Proposition 18]{Hesselholt}. It was observed by Grothendieck that a $\lambda$-structure on a commutative ring $A$ can be repackaged into a ring homomorphism $A \to W(A)$ respecting the comonad structure. Stated more formally, the category $\lambdaCR$ is equivalent to the category $\Coalg(W)$ of coalgebras over $W$. The analogous result holds true for
$\delta$-rings and $p$-typical Witt vectors and has been shown in the animated setting by Bhatt and Lurie. In this section, we prove the equivalence $\lambdaCRan \simeq \Coalg(W)$ for animated $\lambda$-rings (\cref{Thm.: animated lambda-rings are coalgebras over W}). Most of the statements and their proofs carry over verbatim from \cite{BL}.

\begin{remark}
	There exist two equivalent versions of global Witt vectors which we don't distinguish notationally in this section. The first, which we refer to in the remainder of this section, is given by the big Witt ring $W_S(A)$ associated to the
	truncation set $S = \Nmult$, see \cite{Hesselholt}.
	The second version $W^{\Lambda}(A)$ is given by the abelian group $1+tA[[t]]$ of power series over $A$ with constant coefficient 1. Ring multiplication and the comonad structure on $W^{\Lambda}(A)$ are defined in terms of the universal
	polynomials appearing in the definition of a $\lambda$-ring. It is not difficult to see that the equivalence $\lambdaCR \simeq \Coalg(W)$ is best established using either Grothendieck's definition of a $\lambda$-ring and $W^\Lambda$ or
	Joyal's notion of a global $\delta$-ring together with $W_{\Nmult}$. The equivalence of both versions on underlying abelian groups $W_{\Nmult}(A) \simeq W^\Lambda(A)$ has been shown in \cite[Addendum 15]{Hesselholt}. This
	upgrades to an equivalence of comonads.
\end{remark}

\begin{notation}[Restriction]\label{Notation: restriction maps}
	We write $S_n$ for the truncation set $\{1,\ldots,n\}$. For any $n \geq 1$, the inclusion $S_n \subseteq S_{n+1}$ of truncation sets induces a natural transformation $W_{S_{n+1}} \to W_{S_n}$ of endofunctors on
	$\CR$ which we denote by $\res_n$, see \cite{Hesselholt}.
\end{notation}

\begin{notation}[Animated $n$-truncated Witt vectors]\label{Animated n-truncated Witt vectors}
	Let $n \geq 1$. We write $W_n^{\an} \colon \CRan \to \CRan$ for the animation of the functor $W_{S_n}$ of $n$-truncated global Witt vectors. We denote the animation of the restriction
	from \cref{Notation: restriction maps} again by $\res_n$.
\end{notation}

\begin{proposition}[Compare {\cite[A.3 and A.4]{BL}}]\label{Prop.: properties animated n-truncated Witt vectors}
	The functors $W_n^{\an}$ of \cref{Animated n-truncated Witt vectors} satisfy the following properties:
	\begin{enumerate}
	\item For any $n \geq 1$, there exists an equivalence $[W_n^{\an}(A)] \simeq [A]^n$ in $\An$ which is natural in $A \in \CRan$.
	\item For any $n \geq 1$, there exists an equivalence $W_n^{\an}(A) \simeq W_{S_n}(A)$ in $\CRan$ which is natural in $A \in \CR$.
	\end{enumerate}
\end{proposition}

\begin{proof}
	See the proofs of \cite[A.3 and A.4]{BL}.
\end{proof}

\begin{definition}[Global animated Witt vectors]\label{Def.: global animated Witt vectors}
	In view of \cref{Prop.: properties animated n-truncated Witt vectors}, we write $W_n$ for the functors of animated $n$-truncated Witt vectors in the following. We define the functor $W$ of global animated Witt vectors as the
	limit of the diagram
	\begin{equation*}
		\ldots \to W_{n+1} \to W_n \to \ldots \to W_1
	\end{equation*}
	in $\End(\CRan)$ with transition maps given by the restrictions $\res_n$ of \cref{Animated n-truncated Witt vectors}.
\end{definition}

\begin{proposition}\label{Prop.: global Witt vectors are a product}
	There exists an equivalence $[W] \simeq \prod_{n \geq 1} [-]$ of functors $\CRan \to \An$.
\end{proposition}

\begin{proof}
	By \cref{Prop.: properties animated n-truncated Witt vectors}, there exist natural equivalences $[W_n] \simeq [-]^n$ for all $n \geq 1$ which are compatible with 
	restriction: For each animated ring $A$, the diagram
	\[\begin{diag}
		{[W_{n+1}(A)]} \ar[r, "\sim"] \ar[d, "{[\res_n]}"'] & {[A]}^{n+1}\ar[d] \\[3ex]
		{[W_n(A)]} \ar[r, "\sim"] & {[A]}^{n}
	\end{diag}\]
	in $\An$ commutes. Here, the vertical map on the right is induced by the first $n$ projections. The claim thus follows from the definition of $W$ together with the fact that $[-] \colon \CRan \to \An$ preserves small limits.
\end{proof}

\begin{corollary}[Properties of global animated Witt vectors]\label{Cor.: properties of global animated Witt vectors}\phantom{a}\par
	\begin{enumerate}
	\item The functor $W$ of global animated Witt vectors preserves geometric realizations. In particular, $W \simeq \Lan_{\iota}(\restr{W}{\CR})$ is left Kan extended from its restriction to commutative rings.
	\item The functor $W$ coincides with ordinary global Witt vectors on $\CR$.
	\end{enumerate}
\end{corollary}

\begin{proof}
	The functor $[-] \colon \CRan \to \An$ is conservative and preserves geometric realisations. Thus, by \cref{Prop.: global Witt vectors are a product} it suffices to show that the functor $\prod_{n \geq 1} [-]$ commutes with
	geometric realisations. This holds true by \cref{Rmk.: Underlying spectrum} which gives a factorisation $\prod_{n \geq 1} [-] \simeq \Omega^\infty \circ \prod_{n \geq 1} [-]_{\mathbf{E}_\infty}$ into geometric realisation preserving functors.
	The second part of (1) now follows from \cref{Prop.: Properties of animated rings}. For (2), observe that the equivalences of \cref{Prop.: properties animated n-truncated Witt vectors} (2) are compatible with restriction and hence
	induce an equivalence $W(A) \xrightarrow{\sim} \lim_n \iota W_{S_n}(A)$ in animated rings for any ordinary commutative ring $A$. The claim follows since the inclusion $\iota \colon \CR \to \CRan$ preserves small limits and
	$W_{\Nmult}(A) \simeq \lim_n W_{S_n}(A)$.
\end{proof}

\begin{theorem}[Compare {\cite[A.23]{BL}}]\label{Thm.: animated lambda-rings are coalgebras over W}\phantom{a}\par
	\begin{enumerate}
	\item The forgetful functor $\forget_\lambda \colon \lambdaCRan \to \CRan$ admits a right adjoint $\cofreelambda$ which preserves geometric realisations.
	\item There exists an equivalence $\lambdaCRan \simeq \Coalg({\forget_\lambda} \circ \cofreelambda)$ of $\infty$-categories.
	\item The functor underlying the comonad ${\forget_\lambda} \circ \cofreelambda$ on $\CRan$ is equivalent to the functor $W$ of global animated Witt vectors.
	\end{enumerate}
\end{theorem}

\begin{proof}
	Existence of the right adjoint in (1) follows from the adjoint functor theorem. To see that $\cofreelambda$ preserves geometric realisations, recall that $\lambdaCRan$ is generated
	under colimits by the elements of $\lambdaPoly$ which are compact and projective. The claim now follows from observing that the left adjoint $\forget_\lambda$ sends any $X$ in $\lambdaPoly$ to a projective object in $\CRan$.
	Assertion (2) follows from the $\infty$-categorical Barr-Beck-Theorem.\\
	
	We give the argument presented in \cite[A.23]{BL} for (3): By (1), the composite $W' = {\forget_\lambda} \circ \cofreelambda$ preserves geometric realisations and is thus 
	left Kan extended from its restriction to commutative rings. By \cref{Cor.: properties of global animated Witt vectors}, it suffices to produce a natural equivalence $\beta \colon W \circ \iota \xrightarrow{\sim} {W'} \circ \iota$ of
	functors $\CR \to \CRan$.\\
	Given a commutative ring $R$, denote by $\alpha_R \colon \iota W(R) \to \cofreelambda(\iota R)$ the map of animated $\lambda$-rings which is adjoint to the restriction
	$\iota(\res) \colon \iota W(R) \to \iota(R)$ in $\CRan$. Here, we use that ordinary global Witt vectors carry a natural $\lambda$-ring structure and that forgetting (animated) $\lambda$-structures commutes with the
	inclusion into animated ($\lambda$-) rings.
	Let $\beta_R \colon W\iota (R) \simeq \forget_\lambda(\iota W(R)) \to W'\iota(R)$
	denote the map in $\CRan$ induced $\forget_\lambda(\alpha_R)$. Here, the first equivalence follows from \cref{Cor.: properties of global animated Witt vectors} (2). It is clear that the assignment
	$R \mapsto \beta_R$ is natural in $R$ and defines an equivalence if and only if $\alpha_R$ is an equivalence for any commutative ring $R$. This follows if the map
	\begin{equation*}
		\Map_{\lambdaCRan}(A,\iota W(R)) \to \Map_{\lambdaCRan}(A,\cofreelambda(\iota R)) \xrightarrow{\sim} \Map_{\CRan}(A,\iota R)
	\end{equation*}
	induced by $\alpha_R$ is an equivalence for any $A \in \lambdaCRan$. \Cref{Prop.: pi0 is left adjoint to inclusion of lambdaCR} shows that the above map is equivalent to the map
	\begin{equation*}
		\Map_{\lambdaCR}(\pi_0(A),W(R)) \to \Map_{\CR}(\pi_0(A),R)
	\end{equation*}
	induced by the restriction $\res_R \colon W(R) \to R$. This defines an equivalence since $W(R)$ is the ordinary cofree $\lambda$-ring with counit given by $\res_{(-)}$.
\end{proof}

%\newpage
%\appendix

%\section{Ordinary $\lambda$-rings}

\newpage

\bibliography{sources-animated-lambda}{}
\bibliographystyle{plain}

\end{document}